\documentclass[a4paper,11pt]{article}
\usepackage{authblk}
\usepackage[margin=25mm]{geometry}
\usepackage{amsmath,amsthm,amssymb,mathtools,dsfont}
\usepackage[shortlabels]{enumitem}
\usepackage{graphicx}
\usepackage[labelformat=simple]{subcaption} 
\usepackage[noadjust]{cite}
\usepackage{url}

\numberwithin{equation}{section}

\newcommand{\abs}[1]{\left\lvert#1\right\rvert}

\newcommand{\set}[1]{\left\{#1\right\}}

\newcommand{\prob}{\mathbf P}
\newcommand{\ex}{\mathbf E}
\newcommand{\ind}{\mathds 1}
\newcommand{\entropy}{\mathbf H}

\usepackage{marginnote,xcolor}

\theoremstyle{definition}
\newtheorem{definition}{Definition}[section] 
\theoremstyle{plain}
\newtheorem{theorem}[definition]{Theorem}
\newtheorem{proposition}[definition]{Proposition}
\newtheorem{lemma}[definition]{Lemma}
\newtheorem{corollary}[definition]{Corollary}
\theoremstyle{remark}
\newtheorem{remark}[definition]{Remark}

\title{Entropies of the Poisson distribution as functions of intensity: ``normal'' and ``anomalous'' behavior}

\author[1]{Dmitri Finkelshtein\thanks{{d.l.finkelshtein@swansea.ac.uk}}}
\author[2]{Anatoliy Malyarenko\thanks{{anatoliy.malyarenko@mdu.se}}}
\author[3]{Yuliya Mishura\thanks{{yuliyamishura@knu.ua}}}
\author[3,4]{Kostiantyn Ralchenko\thanks{{kostiantynralchenko@knu.ua}}}

\affil[1]{Department of Mathematics, Swansea University, Swansea, SA1 8EN, U.K.}
\affil[2]{Mälardalen University, 721 23 V\"aster{\aa}s, Sweden}
\affil[3]{Department of Probability Theory and Mathematical Statistics, Taras Shevchenko National University of Kyiv, 64 Volodymyrska, 01601, Kyiv, Ukraine}
\affil[4]{School of Technology and Innovations, University of Vaasa, P.O. Box 700, Vaasa, FIN-65101, Finland}

\begin{document}

\maketitle

\begin{abstract}
The paper extends the analysis of the entropies of the Poisson distribution with parameter~$\lambda$. It demonstrates that the Tsallis and Sharma--Mittal entropies exhibit monotonic behavior with respect to $\lambda$, whereas two generalized forms of the R\'enyi entropy may exhibit ``anomalous'' (non-monotonic) behavior. Additionally, we examine the asymptotic behavior of the entropies as $\lambda \to \infty$ and provide both lower and upper bounds for them.

\bigskip


\end{abstract}

\section{Introduction}

Entropy is one of the most debated and paradoxical concepts in science. It is generally linked to the disorder and uncertainty of dynamical systems. At equilibrium, a system's entropy is maximized, obscuring initial conditions. Nowadays, entropy is a fundamental concept across a wide range of disciplines: statistical mechanics \cite{pathria2011statistical}, information theory \cite{CoverThomas}, quantum computation \cite{Nielsen}, cryptography \cite{schneier2015}, biology \cite{xu2020diversity}, neuroscience \cite{Fagerholm2023}, ecological modeling \cite{roach2019emergent}, financial market analysis and portfolio optimization \cite{Zhou-review,giunta2024exploring,MRZZ_2024}, decision tree construction \cite{Aning2022}, machine learning \cite{barbiero2022entropy}, and others. In these contexts, entropy serves as a metric to quantify information, predictability, complexity, and other relevant characteristics.

Initially, the concept of entropy was developed within the framework of thermodynamics and statistical physics. Then, Claude Shannon, in the seminal paper \cite{Shannon}, studied the concept of entropy in the context of information theory. He defined entropy for a discrete probability distribution $\{p_i\}$ as $\sum p_i\log_2 p_i$, establishing it as a limit for data compression and transmission efficiency. Shannon's work laid the foundation for modern cryptography, artificial intelligence, and communication theory.

Since then, several alternative definitions of entropy have emerged in the literature. These definitions have become increasingly complex and generalized by incorporating additional parameters. Despite their diversity, many of these entropies share fundamental properties originally proposed by Alfred R\'enyi \cite{renyi1960}. Today, R\'enyi entropy \cite{renyi1960}, which depends on an additional parameter $\alpha$, is widely used in quantum information theory and quantum statistical mechanics. One- and two-parameter generalized R\'enyi entropies were introduced by Acz\'el and Dar\'oczy~\cite{aczel-daroczy}. Another notable extension is Tsallis entropy, introduced by Constantino Tsallis \cite{tsallis1988} (see also \cite{tsallis2009,tsallis}), which incorporates a nonextensive parameter to model systems that deviate from the principles of classical statistical mechanics. Sharma--Mittal entropy \cite{sharma-mittal}, which generalizes both Shannon and Tsallis entropies, depends on two parameters, $\alpha$ and $\beta$. It provides a refined approach to controlling the balance between order and disorder within a system. The precise definitions of these entropies are provided below; see Definition \ref{def:entropies}. For other examples of entropies, see, for instance, \cite{Arimoto1971,Belis1968,Havrda1967,kapur1967,Kapur1988,Picard1979,Rathie1970,sharma-taneja75,sharma-taneja77,Varma1966}. These examples propose various methodologies for quantifying information content, complexity, and uncertainty in a wide range of physical systems. For a unified approach to various entropies and their applications in information transmission and statistical concepts, we refer to the book \cite{taneja-book}. Comparisons of different entropies for probability distributions can be found in \cite{Dey2016,Mahdy2017,Majumdar2012,malyarenko2023,Maszczyk2008}. For more recent definitions and results, see, for example, \cite{basit2020,basit2021,trandafir2018} and the references therein.

In \cite{malyarenko2023}, six types of entropy measures for the Gaussian distribution were analyzed, focusing on their interrelationships and properties as functions of their parameters. The findings indicated that, for the normal distribution, all entropies exhibited logical, consistent, and predictable behavior. In particular, each entropy increased with an increase in variance.

In \cite{braiman2024}, it was shown that similar properties hold true for the Shannon and Rényi entropies of the Poisson distribution. In particular, both entropies were found to increase as functions of the Poisson parameter, $\lambda$.

In this paper, we extend the study of entropies for the Poisson distribution, establishing, in particular, that the Tsallis and Sharma--Mittal entropies also exhibit ``normal'' behavior, specifically monotonicity with respect to $\lambda$. In contrast, we demonstrate that two generalized forms of R\'enyi entropy for the Poisson distribution may exhibit ``anomalous'' behavior. This does not imply that these entropies should be disregarded; rather, it serves as a caution to researchers that their interpretation and behavior, when applied to different distributions, may not always be intuitive or predictable. We analyze the asymptotic behavior of the entropies for the Poisson distribution as $\lambda \to \infty$: 
we find leading terms of the asymptotic for the R\'enyi entropy, one of the generalised R\'enyi entropies, and the Tsallis and Sharma–Mittal entropies (the asymptotic for the Shannon entropy is known in the literature, see \cite{boersma}).
We derive upper and lower bounds for the entropies. In particular, for the Shannon entropy, we obtain two-side bounds (for $\lambda>1$) which are agreed with its asymptotic behavior (as $\lambda\to\infty$); for the R\'enyi entropy, we estimate it through the Mittag--Leffler function.

The paper is organized as follows. In Section \ref{sec:prelim}, we recall the definitions of six types of entropy for discrete probability distributions, namely: Shannon, R\'enyi, Tsallis, Sharma--Mittal, and two generalized forms of R\'enyi entropy. We also show that all these entropy measures are strictly positive. Section \ref{sec:computation} provides explicit formulas for the various entropies of a Poisson distribution, and we also study their limiting behavior as $\lambda\to\infty$.
Section \ref{sec:bounds} demonstrates upper and lower bounds for the entropies. 
In Section \ref{sec:monotonicity}, we present results on the monotonicity and non-monotonicity of the entropies. In particular, we explore the ``normal'' behavior of the Tsallis and Sharma–Mittal entropies, contrasted with the ``anomalous'' behavior of the two generalized R\'enyi entropies. These theoretical results are complemented by graphical illustrations. Finally, some technical results and proofs are included in the Appendix.

\section{Preliminaries}
\label{sec:prelim}
\looseness=1 Consider a discrete probability distribution with probabilities $\{p_i, i \geq 0\}$.
From now on we assume that all discrete distributions are non-degenerate in the sense that 
\begin{equation}\label{eq:nondegen}
p_i\in(0,1), \ i\geq0, \qquad \sum_{i = 0}^{\infty}p_i=1.  
\end{equation}
Of course, in general, the number of probabilities can be both finite and countable. Let us introduce six types of entropies that will be considered in this paper (we call them by their common names).

\begin{definition}\label{def:entropies}
Let $\mathtt{p}:=\{p_i, i \geq 0\}$ be a non-degenerate discrete probability distribution.
\begin{enumerate}
\item 
The Shannon entropy of $\mathtt{p}$ is given by the formula
\begin{equation}\label{eq:SH}
\entropy_{SH} (\mathtt{p}) =-\sum_{i = 0}^{\infty}  p_i\log {p_i}  =  \sum_{i = 0}^{\infty}   p_i \log\left(\frac{1}{p_i}\right)  .
\end{equation}
\item
The R\'{e}nyi entropy with parameter $\alpha>0,$ $\alpha\ne1$ of $\mathtt{p}$ is given by the formula
\begin{equation}\label{eq:R}
\entropy_{R}(\alpha, \mathtt{p}) = \frac{1}{1-\alpha} \log
\left(\sum_{i=0}^\infty p_i^{\alpha}\right).
\end{equation}

 \item  The generalized R\'enyi entropy with parameter $\alpha > 0$ of $\mathtt{p}$ is given by the formula
    \begin{equation}\label{eq:GR2}
    \entropy_{GR}(\alpha, \mathtt{p}) = - \frac{\sum_{i = 0}^{\infty} p_i^\alpha   \log{p_i}  }{\sum_{i = 0}^{\infty} p_i^\alpha}
    = \frac{\sum_{i = 0}^{\infty} p_i^\alpha   \log{\frac{1}{p_i}}  }{\sum_{i = 0}^{\infty} p_i^\alpha}.
    \end{equation}

\item
    The generalized R\'enyi entropy with parameters $\alpha > 0$, $\beta > 0$, $\alpha \neq \beta$ of $\mathtt{p}$ is given by the formula
    \begin{equation}\label{eq:GR1}
    \entropy_{GR}(\alpha, \beta, \mathtt{p}) = \frac{1}{\beta - \alpha}  \log\left(\frac{\sum_{i = 0}^{\infty} p_i^\alpha }{\sum_{i = 0}^{\infty} p_i^\beta}\right).
    \end{equation}
    
\item
    The Tsallis entropy with parameter $\alpha > 0$, $\alpha \neq 1$ of $\mathtt{p}$ is given by the formula 
    \begin{equation}\label{eq:T}
    \entropy_T(\alpha, \mathtt{p}) = \frac{1}{1-\alpha} \left( \sum_{i = 0}^{\infty} p_i^\alpha - 1 \right).
    \end{equation}

\item
    The Sharma--Mittal entropy with parameters $\alpha > 0$, $\beta > 0$, $\alpha \neq 1$, $\beta \neq 1 $ of $\mathtt{p}$ is given by the formula
    \begin{equation}\label{eq:SM}
    \entropy_{SM}(\alpha, \beta, \mathtt{p}) = \frac{1}{1-\beta} \left( \left( \sum_{i = 0}^{\infty} p_i^\alpha \right)^{\frac{1- \beta}{1 - \alpha}} - 1 \right).
    \end{equation}
\end{enumerate}
\end{definition}

\begin{remark}
Note that, by \eqref{eq:nondegen},
\begin{align}
    \entropy_{GR}(1, \mathtt{p}) &=\entropy_{SH} (\mathtt{p}),\label{eq:SHhtroughGR}\\ \shortintertext{and also}
    \entropy_{GR}(\alpha, 1, \mathtt{p}) &= \entropy_{GR}(\alpha, \mathtt{p}).\notag
\end{align}
\end{remark}

\begin{proposition}\label{prop:positivity}
    All entropies \eqref{eq:SH}--\eqref{eq:SM} take strictly positive values for all admissible values of the parameters. More precisely, let 
    \[
    \mu(\mathtt{p}):= \max_{i\geq0} p_i \in(0,1),
    \]
then
    \begin{align}
            \label{trivial1}
            \entropy_{SH} (p_i, i\ge0) &\ge -\log \mu(\mathtt{p}) > 0,
            \\
            \label{trivialR}
            \entropy_{R}(\alpha, \mathtt{p})&\ge -\log\mu(\mathtt{p})>0,
            \\
            \label{trivialGR1}
            \entropy_{GR}(\alpha, \mathtt{p})&\ge -\log \mu(\mathtt{p}) >0, 
            \\
            \label{trivialGR2}
            \entropy_{GR}(\alpha,\beta, \mathtt{p})&\ge -\log\mu(\mathtt{p})>0,
            \\\label{trivialT}
            \entropy_{T}(\alpha, \mathtt{p})&\ge  \frac{1}{1-\alpha}\bigl(\mu(\mathtt{p}) ^{\alpha-1} - 1 \bigr)>0, 
            \\\label{trivialSM}
            \entropy_{SM}(\alpha, \beta, \mathtt{p})&\ge \frac{1}{1-\beta}\bigl(\mu(\mathtt{p}) ^{\beta-1} - 1 \bigr)>0.
    \end{align}
\end{proposition}

\begin{proof}
1. First, we note that
\[
\entropy_{SH}(p_i, i\ge0) = -\sum_{i\geq 0}p_i \log p_i
\ge -\sum_{i\geq 0}p_i  \log \mu(\mathtt{p}) 
= -\log\mu(\mathtt{p})>0,
\]
and similarly,
\[
 \entropy_{GR}(\alpha, \mathtt{p}) = - \frac{\sum_{i = 0}^{\infty} p_i^\alpha   \log{p_i}  }{\sum_{i = 0}^{\infty} p_i^\alpha}\geq 
 - \frac{\sum_{i = 0}^{\infty} p_i^\alpha   \log\mu(\mathtt{p}) }{\sum_{i = 0}^{\infty} p_i^\alpha} =-\log\mu(\mathtt{p})>0.
\]

2. Next, we consider the function
\begin{equation}\label{eq:psigeneral}
    \psi(\alpha, \mathtt{p}):= \sum_{i=0}^\infty p_i^\alpha>0.
\end{equation}
For $\alpha>1$, we have that
\begin{equation}
    \psi(\alpha, \mathtt{p})= \sum_{i=0}^\infty p_i p_i^{\alpha-1}\leq (\max_{i\geq0} p_i)^{\alpha-1}  \sum_{i=0}^\infty p_i= \mu(\mathtt{p})^{\alpha-1}.\label{eq:alphage1}
\end{equation}
If $\alpha\in(0,1)$, then the sign of the inequality is reverted (as now $\alpha-1<0$):
\begin{equation}
    \psi(\alpha, \mathtt{p})= \sum_{i=0}^\infty p_i p_i^{\alpha-1}\geq (\max_{i\geq0} p_i)^{\alpha-1}  \sum_{i=0}^\infty p_i= \mu(\mathtt{p})^{\alpha-1}.\label{eq:alphale1}
\end{equation}
Combining \eqref{eq:alphage1}--\eqref{eq:alphale1}, we will get that, for all $\alpha>0$, $\alpha\neq 1$,
\[
\entropy_R(\alpha,\mathtt{p}) = \frac{1}{1-\alpha}\log \psi(\alpha, \mathtt{p})\geq \frac{1}{1-\alpha}\log \mu(\mathtt{p})^{\alpha-1}=-\log\mu(\mathtt{p}),
\]
and
\[
\entropy_T(\alpha,\mathtt{p})= \frac1{1-\alpha} \bigl( \psi(\alpha, \mathtt{p})-1 \bigr)\geq \frac1{1-\alpha}(\mu(\mathtt{p})^{\alpha-1}-1)>0,
\]
as the signs of $1-\alpha$ and $\mu(\mathtt{p})^{\alpha-1}-1$ coincide. 
Also, combining \eqref{eq:alphage1}--\eqref{eq:alphale1}, we will get that, for all $\alpha>0$, $\alpha\neq 1$,
\begin{equation}\label{eq:bgimumin1}
    \bigl(\psi(\alpha, \mathtt{p})\bigr)^{\frac{1}{1-\alpha}}\geq \mu(\mathtt{p})^{-1},
\end{equation}
and therefore,
\[
\entropy_{SM}(\alpha, \beta, \mathtt{p}) = \frac{1}{1-\beta} \left( \left(\psi(\alpha, \mathtt{p}) \right)^{\frac{1- \beta}{1 - \alpha}} - 1 \right)\geq
\frac{1}{1-\beta} \left( \mu(\mathtt{p})^{\beta-1} - 1 \right)>0.
\]

3. Finally, we have
\[
\entropy_{GR}(\alpha,\beta, \mathtt{p})=\frac{1}{\beta-\alpha} \log\left(
\frac{\sum\limits_{i\geq0} p_i^{\beta} p_i^{\alpha - \beta}}{\sum\limits_{i\geq0} p_i^{\beta}}\right).
\]
For $\alpha>\beta$, we have $p_i^{\alpha-\beta}\leq \mathtt{p}^{\alpha-\beta}$, and since $\beta-\alpha<0$, we will get 
\[
\entropy_{GR}(\alpha,\beta, \mathtt{p})\geq \frac{1}{\beta-\alpha}
\log\left(
\frac{\sum\limits_{i\geq0} p_i^{\beta} \mu(\mathtt{p})^{\alpha - \beta}}{\sum\limits_{i\geq0} p_i^{\beta}}\right) =\mu(\mathtt{p}).
\]
For $\alpha<\beta$, $p_i^{\alpha-\beta}\geq \mu(\mathtt{p})^{\alpha-\beta}$ and $\beta-\alpha>0$, hence, we will get the same estimate.
\end{proof}

\begin{remark}
    Note that $\psi(\alpha,\mathtt{p})$ is strictly decreasing in $\alpha$, as $p_i\in(0,1)$. Then, for $\alpha >1$, $\psi(\alpha,\mathtt{p})<\psi(1,\mathtt{p})=1$, and hence,
    \begin{equation}\label{eq:upperT}
        \entropy_T(\alpha,\mathtt{p})=\frac1{\alpha-1} \bigl( 1-\psi(\alpha, \mathtt{p})\bigr)<\frac1{\alpha-1}.
    \end{equation}
    Next, for $\beta>1$, we get from \eqref{eq:bgimumin1} that
    \[
    \bigl(\psi(\alpha, \mathtt{p}) \bigr)^{\frac{1- \beta}{1 - \alpha}}\leq \mu(\mathtt{p})^{\beta-1} <1,
    \]
    and hence,
    \begin{equation}\label{eq:upperSM}
        \entropy_{SM}(\alpha, \beta, \mathtt{p}) = \frac{1}{\beta-1} \left( 1- \left(\psi(\alpha, \mathtt{p}) \right)^{\frac{1- \beta}{1 - \alpha}}  \right)<\frac{1}{\beta-1}.
    \end{equation}
\end{remark}

\section{Entropies of the Poisson distribution}
\label{sec:computation}

We are going to consider properties of the six types of entropies introduced above for Poisson distribution; for other results in this direction, see e.g. \cite{braiman2024,  MR4312787, malyarenko2023}.

Recall that a random variable $X_\lambda$ has a Poisson distribution with parameter $\lambda > 0$ if
\[
p_i(\lambda):=\prob(X_\lambda = i) = \frac{\lambda^i  e^{-\lambda}}{i!},\quad i \ge 0.
\]

In the sequel, we denote by $\entropy_{SH}(\lambda)$, $\entropy_{R}(\alpha, \lambda)$, $\entropy_{GR}(\alpha, \lambda)$, $\entropy_{GR}(\alpha, \beta, \lambda)$, $\entropy_{T}(\alpha, \lambda)$, $\entropy_{SM}(\alpha, \beta, \lambda)$ the entropies \eqref{eq:SH}--\eqref{eq:SM}, respectively, for the Poisson distribution with parameter $\lambda$.

\begin{proposition}\label{prop:entropies-poisson}
\begin{enumerate}[(i)]
\item The Shannon entropy for a Poisson distribution with parameter $\lambda > 0$ equals
\begin{equation}\label{eq:exprSH}
     \entropy_{SH}(\lambda) = -\lambda \log \left( \frac{\lambda}{e} \right) + e^{-\lambda} \sum_{i = 2}^\infty \frac{\lambda^i \log(i!)}{i!}.
\end{equation}

\item
    The R\'enyi entropy with parameter $\alpha > 0$, $\alpha \neq 1$ for a Poisson distribution with parameter $\lambda > 0$ equals
    \[
    \entropy_R(\alpha, \lambda) = \frac{1}{1-\alpha} \log \left( e^{-\alpha \lambda} \sum_{i=0}^{\infty} \frac{\lambda^{i \alpha}}{(i !)^\alpha} \right).
    \]
    
\item The generalized R\'enyi entropy with parameter $\alpha > 0$ for the Poisson distribution with parameter $\lambda > 0$ equals
\[
\entropy_{GR}(\alpha, \lambda)= \frac{\sum_{i = 1}^{\infty} {\frac{\lambda^{i \alpha}} {(i!)^\alpha}}  (\log i! - i  \log\lambda ) }{\sum_{i = 0}^{\infty} {\frac{\lambda^{i \alpha} }{(i!)^\alpha}}} + \lambda.
\]

\item The generalized R\'enyi entropy with parameters $\alpha \neq \beta$, $\alpha > 0$, $\beta > 0$ for the Poisson distribution with parameter $\lambda > 0$ equals
\[
\entropy_{GR}(\alpha, \beta, \lambda) = \lambda + \frac{1}{\beta - \alpha} \log \left( \frac {\sum_{i = 0}^{\infty} \frac{\lambda^{i \alpha}}{(i!)^\alpha}} {\sum_{i = 0}^{\infty} \frac{\lambda^{i \beta}}{(i!)^\beta}} \right).
\]

\item The Tsallis entropy with parameter $\alpha > 0, \alpha \neq 1$ for the Poisson distribution with parameter $\lambda > 0$ equals
\[
\entropy_{T}(\alpha, \lambda) = \frac{1}{1-\alpha} \left( e^{-\lambda \alpha }\sum_{i = 0}^{\infty} \frac{\lambda^{i \alpha} }{(i!)^\alpha} - 1 \right).
\]

\item The Sharma--Mittal entropy with parameters $\alpha > 0$, $\beta > 0$, $\alpha \neq 1$, $\beta \neq 1 $ for the Poisson distribution with parameter $\lambda > 0 $ equals
\[
\entropy_{SM}(\alpha, \beta, \lambda) = \frac{1}{1-\beta} \left( e^\frac{-\lambda  \alpha  (1 - \beta)}{1 - \alpha}\left( \sum_{i = 0}^{\infty} \frac{\lambda^{i \alpha} }{(i!)^\alpha} \right)^{\frac{1- \beta}{1 - \alpha}} - 1 \right).
\]
\end{enumerate}
\end{proposition}
    
\begin{proof}
The claims $(i)$ and $(ii)$ are proved in \cite{braiman2024}.

$(iii)$ By \eqref{eq:GR2},
\begin{align*}
\entropy_{GR}(\alpha, \lambda) &=- \frac{\sum_{i = 0}^{\infty} p_i^\alpha  \log{p_i}  }{\sum_{i = 0}^{\infty} p_i^\alpha}
= - \frac{\sum_{i = 0}^{\infty} {\frac{\lambda^{i \alpha}  e^{-\lambda\alpha}}{(i!)^\alpha}}  \log{\frac{\lambda^i  e^{-\lambda}}{i!}}  }{\sum_{i = 0}^{\infty} {\frac{\lambda^{i \alpha}  e^{-\lambda\alpha}}{(i!)^\alpha}}}
\\
&= - \frac{\sum_{i = 0}^{\infty} {\frac{\lambda^{i \alpha}} {(i!)^\alpha}}  (i  \log\lambda - \lambda - \log i!) }{\sum_{i = 0}^{\infty} {\frac{\lambda^{i \alpha} }{(i!)^\alpha}}} 
= \frac{\sum_{i = 0}^{\infty} {\frac{\lambda^{i \alpha}} {(i!)^\alpha}}  (\log i! - i  \log\lambda ) }{\sum_{i = 0}^{\infty} {\frac{\lambda^{i \alpha} }{(i!)^\alpha}}} + \lambda.
\end{align*}

$(iv)$ By \eqref{eq:GR1},
\begin{align*}
\entropy_{GR}(\alpha, \beta, \lambda) &= \frac{1}{\beta - \alpha} \log\left(\frac{\sum_{i = 0}^{\infty} p_i^\alpha }{\sum_{i = 0}^{\infty} p_i^\beta}\right)
= \frac{1}{\beta - \alpha} \log \left( \frac {\sum_{i = 0}^{\infty} \frac{\lambda^{i \alpha} e^{-\lambda  \alpha }}{(i!)^\alpha}} {\sum_{i = 0}^{\infty} \frac{\lambda^{i \beta} e^{-\lambda  \beta }}{(i!)^\beta}} \right)
\\
&= \frac{1}{\beta - \alpha} \log \left( e^{\lambda  (\beta - \alpha)} \frac {\sum_{i = 0}^{\infty} \frac{\lambda^{i \alpha}}{(i!)^\alpha}} {\sum_{i = 0}^{\infty} \frac{\lambda^{i \beta}}{(i!)^\beta}} \right)
=\lambda + \frac{1}{\beta - \alpha} \log \left( \frac {\sum_{i = 0}^{\infty} \frac{\lambda^{i \alpha}}{(i!)^\alpha}} {\sum_{i = 0}^{\infty} \frac{\lambda^{i \beta}}{(i!)^\beta}} \right).
\end{align*}

$(v)$ By \eqref{eq:T},
\[
\entropy_T(\alpha, \lambda) = \frac{1}{1-\alpha} \left( \sum_{i = 0}^{\infty} p_i^\alpha - 1 \right)
= \frac{1}{1-\alpha} \left( \sum_{i = 0}^{\infty} \frac{\lambda^{i \alpha}  e^{-\lambda \alpha }}{(i!)^\alpha} - 1 \right)
=\frac{1}{1-\alpha} \left( e^{-\lambda \alpha }\sum_{i = 0}^{\infty} \frac{\lambda^{i \alpha} }{(i!)^\alpha} - 1 \right).
\]

$(vi)$ By \eqref{eq:SM},
\begin{align*}
    \entropy_{SM}(\alpha, \beta, \lambda) &= \frac{1}{1-\beta} \left( \left( \sum_{i = 0}^{\infty} p_i^\alpha \right)^{\frac{1- \beta}{1 - \alpha}} - 1 \right)
    = \frac{1}{1-\beta} \left( \left( \sum_{i = 0}^{\infty} \frac{\lambda^{i \alpha}  e^{-\lambda  \alpha}}{(i!)^\alpha} \right)^{\frac{1- \beta}{1 - \alpha}} - 1 \right)
    \\
    &= \frac{1}{1-\beta} \left( e^\frac{-\lambda  \alpha  (1 - \beta)}{1 - \alpha}\left( \sum_{i = 0}^{\infty} \frac{\lambda^{i \alpha} }{(i!)^\alpha} \right)^{\frac{1- \beta}{1 - \alpha}} - 1 \right).\qedhere
\end{align*}
\end{proof}

Let us introduce the function, cf. \eqref{eq:psigeneral},
\begin{equation}\label{psifunction}
\psi(\alpha,\lambda) = \sum_{i = 0}^{\infty}\bigl(p_i(\lambda)\bigr)^\alpha = e^{-\lambda \alpha }\sum_{i = 0}^{\infty} \frac{\lambda^{i \alpha} }{(i!)^\alpha},
\quad \lambda > 0,\; \alpha > 0.
\end{equation}
By Lemma~\ref{l:der-psi} in Appendix, function $\psi(\alpha,\lambda)$ is well-defined and continuously differentiable in both variables $\lambda > 0$ and $\alpha > 0$.

\begin{corollary}
All considered entropies for the Poisson distribution can be then expressed in the terms of function $\psi$:
    \begin{align}
    \entropy_R(\alpha,\lambda)&=\frac{1}{1-\alpha}\log \psi(\alpha,\lambda);\label{eq:entR-psi}\\
    \entropy_{T}(\alpha,\lambda)&=\frac{1}{1-\alpha}\bigl(\psi(\alpha,\lambda)-1\bigr);\label{eq:entT-psi}\\ 
    \entropy_{GR}(\alpha,\beta,\lambda)&=\frac1{\beta-\alpha} \log \frac{\psi(\alpha, \lambda)}{\psi(\beta, \lambda)};\label{eq:entGGR-psi}\\
    \entropy_{SM}(\alpha,\beta,\lambda)&=
    \frac{1}{1-\beta}\left( \bigl(\psi(\alpha,\lambda)\bigr)^{\frac{1-\beta}{1-\alpha}}-1\right);\label{eq:entSM-psi}\\
    \entropy_{GR}(\alpha,\lambda)&=-\frac{\partial}{\partial\alpha} \log \psi(\alpha, \lambda);\label{eq:entGR-psi}\\
    \entropy_{SH}(\lambda)&=-\left.\frac{\partial}{\partial\alpha} \log \psi(\alpha, \lambda)\right\rvert_{\alpha=1}.\label{eq:entSH-psi}
    \end{align}
\end{corollary}

\begin{proof}
Formulas \eqref{eq:entR-psi}--\eqref{eq:entSM-psi} follow directly from Definiton~\ref{def:entropies}, Proposition~\ref{prop:entropies-poisson}, and formula \eqref{psifunction}.
To prove \eqref{eq:entGR-psi}, we note that, by \eqref{eq:der-psi-alpha} from Appendix, we get, for $p_i(\lambda) = e^{-\lambda} \frac{\lambda^i}{i!}$,
\[
\frac{\partial}{\partial\alpha} \log \psi(\alpha, \lambda)
= \frac{\frac{\partial}{\partial\alpha} \psi(\alpha, \lambda)}{\psi(\alpha, \lambda)}
= \frac{\sum_{i=0}^\infty (p_i(\lambda))^\alpha\log p_i(\lambda)}{\sum_{i=0}^\infty (p_i(\lambda))^\alpha}
= -\entropy_{GR}(\alpha,\lambda),
\]
by the definition of $\entropy_{GR}$. Finally, \eqref{eq:entSH-psi} follows from \eqref{eq:entGR-psi} and \eqref{eq:SHhtroughGR}.
\end{proof}

The following proposition combines the monotonicity property of $\psi(\alpha,\lambda)$, established previously in \cite{braiman2024}, with one-side estimates and the asymptotic in $\lambda\to\infty$.

\begin{proposition}\label{prop:psi-monot}
\leavevmode
\begin{enumerate}[(i)]
    \item  For every $0 < \alpha < 1$, the function $\psi(\alpha,\lambda)$ strictly increases as a function of $\lambda$ on $(0,\infty)$, and 
    \[
    \psi(\alpha,\lambda)\geq \bigl(2\pi\lfloor\lambda\rfloor\bigr)^{\frac{1-\alpha}{2}}, \quad\lambda\geq1.
    \]
    \item 
    For every $\alpha > 1$, the function $\psi(\alpha,\lambda)$ strictly decreases as a function of $\lambda$ on $(0,\infty)$, and
     \[
    \psi(\alpha,\lambda)\leq \bigl(2\pi\lfloor\lambda\rfloor\bigr)^{-\frac{\alpha-1}{2}}, \quad\lambda\geq1.
    \]
    \item For all $\alpha>0$,
    \begin{equation}\label{eq:psi-asymp-main}
    \psi(\alpha,\lambda)\sim \frac{1}{ \sqrt{\alpha}}(2\pi\lambda)^{\frac{1-\alpha}{2}}, \qquad\lambda\to\infty.
    \end{equation}
\end{enumerate}
\end{proposition}
\begin{proof}   
The monotonicity of $\psi$ was proved in \cite[Theorem 2]{braiman2024}. Next, by Lemma~\ref{l:max-poisson}, for $\mathtt{p}=\{p_i(\lambda),i\geq0\}$,
\begin{equation}\label{eq:additional}
    \mu(\mathtt{p})=\max_{i\ge 0} p_i(\lambda)\leq \frac{1}{\sqrt{2\pi\lfloor\lambda\rfloor}}, \qquad \lambda\ge1. 
\end{equation}
Then the one-side bounds for $\psi(\alpha,\lambda)$ follow from \eqref{eq:alphage1}--\eqref{eq:alphale1}. Finally, the asymptotic \eqref{eq:psi-asymp-main} is proven in Lemma~\ref{le:psi-asymp}.
\end{proof}

\begin{theorem} \label{thm:limits}
\begin{itemize}
    \item[$(a)$]  The Shannon, R\'enyi and both generalised R\'enyi entropies of the Poisson distribution, for all admissible values of the parameters, converge to infinity as $\lambda$ tends to infinity. Moreover, for $\lambda\to\infty$,
    \begin{align}
    \entropy_{SH}(\lambda)&\sim \frac{1}{2}\log(2\pi\lambda)+\frac12;\label{eq:asySH}\\
    \entropy_R(\alpha,\lambda)&\sim \frac{1}{2}\log(2\pi\lambda)+\frac{\log\alpha}{2(\alpha-1)};\label{eq:asyR}\\
    \entropy_{GR}(\alpha,\beta,\lambda)&\sim \frac{1}{2}\log(2\pi\lambda)+\frac{\log\beta-\log\alpha}{2(\beta-\alpha)}.\label{eq:asyGR1}
\end{align}
    \item [$(b)$]  The Tsallis entropy of the Poisson distribution converges to infinity, as $\lambda$ tends to infinity, only for $\alpha \in (0,1)$, more precisely, 
    \[
    \entropy_{T} (\alpha, \lambda)\sim \frac{1}{ \sqrt{\alpha}(1-\alpha)}(2\pi\lambda)^{\frac{1-\alpha}{2}}, \qquad\lambda\to\infty, \ \alpha\in(0,1),
    \]
    and
    \[
    \lim_{\lambda\to\infty} \entropy_{T} (\alpha, \lambda) =
    \dfrac{1}{\alpha - 1}, \qquad \alpha > 1.
    \]
    \item [$(c)$]  The Sharma--Mittal entropy of the Poisson distribution converges to infinity, as $\lambda$ tends to infinity, only for $\beta \in (0,1)$, more precisely,  
    \[
        \entropy_{SM} (\alpha,\beta, \lambda)\sim \frac{1}{(1-\beta)\alpha^{\frac{1-\beta}{2(1-\alpha)}}}(2\pi\lambda)^{\frac{1-\beta}{2}}, \qquad\lambda\to\infty, \ \beta\in(0,1),
    \]
    and
    \[
    \lim_{\lambda\to\infty} \entropy_{SM} (\alpha, \beta, \lambda) =
    \dfrac{1}{\beta - 1}, \qquad \beta > 1.
    \]
    \end{itemize}
\end{theorem}
\begin{proof}
    \begin{itemize}
        \item[$(a)$] By \eqref{eq:additional},
        $-\log\mu(\mathtt{p})\geq \frac12\log\bigl(2\pi\lfloor\lambda\rfloor\bigr)\to\infty$, as $\lambda\to\infty$. Therefore, by \eqref{trivial1}--\eqref{trivialGR2}, 
          \begin{equation*}
      \entropy (\lambda)\geq \frac12\log\bigl(2\pi\lfloor\lambda\rfloor\bigr)\to\infty, \qquad \lambda\to\infty,
  \end{equation*}
  where $\entropy (\lambda)$ is either of $\entropy_{SH}(\lambda)$, $\entropy_R(\alpha,\lambda)$, $\entropy_{GR}(\alpha,\lambda)$, $\entropy_{GR}(\alpha,\beta,\lambda)$. The asymptotic \eqref{eq:asySH} was shown e.g. in \cite{boersma}. The asymptotic \eqref{eq:asyR} follows immediately from \eqref{eq:entR-psi} and \eqref{eq:psi-asymp-main}; note that $\frac{\log\alpha}{2(\alpha-1)}>0$ for each $\alpha>0$, $\alpha\neq1$. Finally, by \eqref{eq:entGGR-psi} and \eqref{eq:psi-asymp-main},
  \[
  \entropy_{GR}(\alpha,\beta,\lambda)=\frac1{\beta-\alpha} \log \frac{\psi(\alpha, \lambda)}{\psi(\beta, \lambda)}\sim \frac1{\beta-\alpha}\log\left(\sqrt{\frac{\beta}{\alpha}}(2\pi\lambda)^{\frac{\beta-\alpha}{2}}\right),
  \]
  that yields \eqref{eq:asyGR1}; note that $\frac{\log\beta-\log\alpha}{2(\beta-\alpha)}>0$ for all $\alpha,\beta>0$, $\alpha\neq\beta$.
        \item[$(b)$] For $\alpha\in(0,1)$, the statement follows from \eqref{trivialT} and \eqref{eq:additional}, as then 
        \begin{equation}\label{eq:powergrowthT}
            \entropy_{T}(\alpha, \lambda )\ge  \frac{1}{1-\alpha}\Bigl(\bigl(2\pi\lfloor\lambda\rfloor\bigr)^{\frac{1-\alpha}{2}} -1\Bigr)\to\infty, \quad \lambda\to\infty.
        \end{equation}
        The asymptotic follows immediately from \eqref{eq:entT-psi} and \eqref{eq:psi-asymp-main}.
        
        For $\alpha>1$, we get from \eqref{trivialT}, \eqref{eq:upperT}, and \eqref{eq:additional}, that
        \begin{equation}\label{eq:powergrowthT2}
            \frac{1}{\alpha-1}>\entropy_{T}(\alpha, \lambda )\ge  \frac{1}{\alpha-1}\Bigl(1-\bigl(2\pi\lfloor\lambda\rfloor\bigr)^{\frac{1-\alpha}{2}} \Bigr)\to \dfrac{1}{\alpha-1}, \quad \lambda\to\infty,
        \end{equation}
        that implies the statement.

        \item[$(c)$] Similarly to $(b)$, by using \eqref{trivialSM} and, for $\beta>1$, \eqref{eq:upperSM}, we will get that, for any $\alpha>0$, $\alpha\neq1$,
        \begin{equation}\label{eq:powergrowthSM}
            \entropy_{SM}(\alpha, \beta, \lambda )\ge  \frac{1}{1-\beta}\Bigl(\bigl(2\pi\lfloor\lambda\rfloor\bigr)^{\frac{1-\beta}{2}} -1\Bigr)\to\infty, \quad \lambda\to\infty,\  \beta\in(0,1),
        \end{equation}
        with the asymptotic coming from \eqref{eq:entSM-psi} and \eqref{eq:psi-asymp-main}; and
        \begin{equation}\label{eq:powergrowthSM2}
            \frac{1}{\beta-1}>\entropy_{SM}(\alpha, \beta, \lambda )\ge  \frac{1}{\beta-1}\Bigl(1-\bigl(2\pi\lfloor\lambda\rfloor\bigr)^{\frac{1-\beta}{2}} \Bigr)\to \dfrac{1}{\beta-1}, \quad \lambda\to\infty,\  \beta>1, 
        \end{equation}
        that proves the statement.
    \end{itemize}
\end{proof}

\section{Lower and upper bounds for entropies of the Poisson distribution}
\label{sec:bounds}

Theorem~\ref{thm:limits} shows that the Shannon, R\'enyi and both generalised R\'enyi entropies of the Poisson distribution grow logarithmically at infinity, whereas, the Tsallis entropy $\entropy_{T}(\alpha, \lambda )$ (for $0<\alpha<1$) and the Sharma--Mittal entropy $\entropy_{SM} (\alpha, \beta, \lambda)$ (for $0<\beta<1$) have the power growth to infinity. In this section we discuss 
double-side estimates for the entropies.

\begin{theorem}
\begin{enumerate}
    \item 
 For any $\lambda>1$,
 \begin{equation}\label{eq:lowerboundEnt}
     \entropy(\lambda)\geq \frac{1}{2}\log (2\pi\lambda) -h(\lambda)=:L(\lambda), 
 \end{equation}
where $\entropy (\lambda)$ is either of $\entropy_{SH}(\lambda)$, $\entropy_R(\alpha,\lambda)$, $\entropy_{GR}(\alpha,\lambda)$, $\entropy_{GR}(\alpha,\beta,\lambda)$ for all admissible values of the parameters and
\begin{equation}
        h(\lambda):= \frac{1}{2}\log\left(1+\frac{1}{(\lambda-1)\vee 1}\right)-\frac1{12\lambda+1}>0, \qquad \lambda>0.\label{eq:defh}
\end{equation}
 \item  For the Shannon entropy, the following upper bound holds
\begin{equation}
    \label{eq:sh-bound}
\entropy_{SH}(\lambda)\le    \frac{1}{2} \log {(2\pi \lambda)}+1+\frac{1}{6\lambda}=:U_{SH}(\lambda), \qquad\lambda>1.
\end{equation}

\item For the R\'enyi entropy, we define $\gamma_*:=\exp\left(-\dfrac{\pi}{e}\bigl(e^{\frac16}-1\bigr)\right)\approx 0.811$, and pick any $\gamma\in[\gamma_*,1)$. Then $\entropy_{R}(\alpha,\lambda)<U_{R}(\alpha,\lambda,\gamma)$ for all $\lambda>1$ and $\alpha>0$, $\alpha\neq1$, where:
\begin{enumerate}
    \item for $\alpha\in(0,1)$ and $\lambda>1$,
\begin{align}\notag
    U_{R}(\alpha,\lambda,\gamma):=&\frac{1}{1-\alpha}\Biggl(\log E_\alpha\left(\Bigl(\dfrac{\alpha}{\gamma}\lambda\Bigr)^\alpha\right)-\alpha\lambda\Biggr)+\frac{1}{2}\log\frac{\pi}{-e\log\gamma}
    \\&\quad +\frac\alpha{2(1-\alpha)}\log\alpha +\frac{1}{12\alpha(1-\alpha)}+\frac{1}{2}\log(1-\alpha).
    \label{eq:estR-up1}
\end{align}
\item for $\alpha>1$ and $\lambda>1$,
\begin{align}\notag
    U_{R}(\alpha,\lambda,\gamma):&=\frac{1}{\alpha-1}\Biggl(\alpha\lambda-\log E_\alpha\bigl( (\alpha\gamma\lambda)^\alpha \bigr)\Biggr)+\frac12\log\frac{\pi}{-e\log\gamma}\\&\quad -\frac{\alpha}{2(\alpha-1)}\log\alpha+\frac{\alpha}{12(\alpha-1)}+\frac12\log(\alpha-1).\label{eq:estR-up2}
\end{align}
\end{enumerate}
Here $E_\alpha$ denotes the Mittag--Leffler function.
\item For the Tsallis entropy, the estimates \eqref{eq:powergrowthT}
and \eqref{eq:powergrowthT2} hold true with $\bigl(2\pi\lfloor\lambda\rfloor\bigr)^{\frac{1-\alpha}{2}}$ replaced by $(2\pi\lambda)^{\frac{1-\alpha}{2}}e^{-h(\lambda)}$.
\item For the Sharma--Mittal entropy, the estimates \eqref{eq:powergrowthSM}
and \eqref{eq:powergrowthSM2} hold true with $\bigl(2\pi\lfloor\lambda\rfloor\bigr)^{\frac{1-\beta}{2}}$ replaced by $(2\pi\lambda)^{\frac{1-\beta}{2}}e^{-h(\lambda)}$.
\end{enumerate}
\end{theorem}
\begin{proof}
1) The lower bounds for $\entropy_{SH}(\lambda)$, $\entropy_R(\alpha,\lambda)$, $\entropy_{GR}(\alpha,\lambda)$, $\entropy_{GR}(\alpha,\beta,\lambda)$ follow immediately from the general estimates \eqref{trivial1}--\eqref{trivialGR2} and the upper bound \eqref{eq:mulambdale1overlaenhanced} for $\mu(\mathtt{p})=\mu(\lambda)$ which reads as follows:
\begin{equation}
    \mu(\lambda)<\frac{1}{\sqrt{2\pi\lambda}} e^{h(\lambda)}, \qquad \lambda>1, \label{eq:muenhanced}
\end{equation}
where $h(\lambda)>0$ is given by \eqref{eq:defh}.

2) For the upper bound for $\entropy_{SH}(\lambda)$, we note that \eqref{eq:exprSH} trivially implies
\begin{equation}\label{eq:trivestSH}
    \entropy_{SH}(\lambda) \le -\lambda  \log  \frac{\lambda}{e} +\sum _{k=1}^{\infty} \left( e^{-\lambda} \frac{\lambda^k}{k!} \right) \log{k!}
\end{equation}
According to the second inequality in \eqref{eq:stirlingest}, for any $k\ge 1$,
\begin{equation*}
\log{k!} \le \frac{1}{2} \log{(2\pi k)} +k\log{\frac{k}{e}} + \frac{1}{12k}.
\end{equation*}

Recall that $X_\lambda$ denotes a random variable having Poisson distribution with parameter $\lambda$. Then, by Jensen's inequality,
\begin{align*}
\frac{1}{2}\sum_{k=1}^{\infty}\left(e^{-\lambda}\frac{\lambda^k}{k!}\right)\log {(2\pi k)} 
&=\frac{1}{2} \ex \log\left(2\pi X_\lambda \ind_{X_\lambda \ge 1}\right) \\& \le \frac{1}{2} \log \ex \left(2\pi X_\lambda \ind_{X_\lambda \ge 1}\right) \le \frac{1}{2}\log \ex (2\pi X_\lambda) = \frac{1}{2} \log{(2\pi \lambda)}
\end{align*}
Similarly, 
\begin{align*}
\sum_{k=1}^{\infty} \left(e^{-\lambda}\frac{\lambda^k}{k!}\right)k\log {\frac{k}{e}} &= \sum_{k=0}^{\infty}\left(e^{-\lambda}\frac{\lambda^{k+1}}{k!}\right)\log {\frac{k+1}{e}} \\ 
&=  \lambda \ex  \log \frac{X_\lambda +1}{e} \le \lambda \log \ex \frac{X_\lambda+1}{e} = \lambda \log \frac{\lambda+1}{e};
\end{align*}
and finally, 
\begin{align*}
  \sum_{k=1}^{\infty} \left(e^{-\lambda}\frac{\lambda^k}{k!}\right) \frac{1}{12k}<\frac16\sum_{k=1}^{\infty} \left(e^{-\lambda}\frac{\lambda^k}{k!}\right) \frac{1}{k+1} = \frac16\frac{e^{-\lambda}}{\lambda}\sum_{k=2}^{\infty}\frac{\lambda^k}{k!}<\frac{1}{6\lambda}.
\end{align*}
Combining these estimates altogether, we get from \eqref{eq:trivestSH}:
\begin{align*}
   \entropy_{SH}(\lambda) &< - \lambda \log \frac{\lambda}{e} + \frac{1}{2} \log(2\pi \lambda) + \lambda \log \frac{\lambda+1}{e}  + \frac{1}{6\lambda}\notag\\ & =\frac{1}{2}\log(2\pi \lambda)+\log \left(1+\frac{1}{\lambda}\right)^{\lambda}  + \frac{1}{6\lambda} \notag\\ &< \frac{1}{2} \log {(2\pi \lambda)}+1+ \frac{1}{6\lambda},
\end{align*}
that fulfills the proof.

3) The upper bound for $\entropy_{R}(\lambda)$  follows immediately from \eqref{eq:entR-psi} and Lemma~\ref{le:estforpsi}.

4) The statement follows from \eqref{trivialT} and \eqref{eq:muenhanced}.

5) The statement follows \eqref{trivialSM} and \eqref{eq:muenhanced}.
\end{proof}

\begin{remark}
We are going to illustrate the obtained estimates. 
\begin{enumerate}
    \item The lower and upper bounds \eqref{eq:lowerboundEnt} and 
\eqref{eq:sh-bound} for the Shannon entropy are agreed with the leading term $\frac12\log(2\pi\lambda)$ of the asymptotics \eqref{eq:asySH}.
Surprisingly, the asymptotic function $A_{SH}(\lambda)$ defined by the right hand side of \eqref{eq:asySH} provides a perfect approximation for the Shannon entropy for \emph{all} $\lambda>1$, see Figure~\ref{fig:estimShannon}.

\item The upper estimates \eqref{eq:estR-up1} for the R\'enyi entropy is not asymptotically exact. Indeed, for e.g. $\alpha\in(0,1)$,  $E_\alpha(x)\sim \frac{1}{\alpha}\exp\bigl(x^\frac{1}{\alpha}\bigr)$, $x\to\infty$, therefore, the function $U_{R}(\alpha,\lambda,\gamma)$ given by \eqref{eq:estR-up1} has the following asymptotic as $\lambda\to\infty$
\[
\frac{\alpha}{1-\alpha}\left(\frac{1}{\gamma}-1\right)\lambda+C(\alpha,\gamma),
\]
that grows faster than the asymptotic function $A_{R}(\lambda)$ defined by the right hand side of \eqref{eq:asyR}. 
Note also that $C(\alpha,\gamma)$ grows for $\alpha$ becoming close either to $0$ or to $1$. It can be shown numerically, that for each $\alpha$ and for each finite interval $[1,\lambda_1[$ there exists an optimal $\gamma_1\in[\gamma_*,1)$ which minimizes $\sup_{\lambda\in[1,\lambda_1]} \bigl(U_{R}(\alpha,\lambda,\gamma_1)-\entropy_R(\alpha,\lambda)\bigr)$.  Again, surprisingly, the asymptotic function $A_{R}(\alpha, \lambda)$ provides a perfect approximation for the R\'enyi entropy for \emph{all} $\lambda>1$, see Figure~\ref{fig:estimRenyi}. As expected, the upper bound is getting worth for $\alpha$ close to $0$ and $1$. Also, in contrast to the upper bound for the Shannon entropy, it will become less reliable with the growth of $\lambda$.
\end{enumerate}
\end{remark}

\begin{figure}
\centering \includegraphics[width=0.6\textwidth]{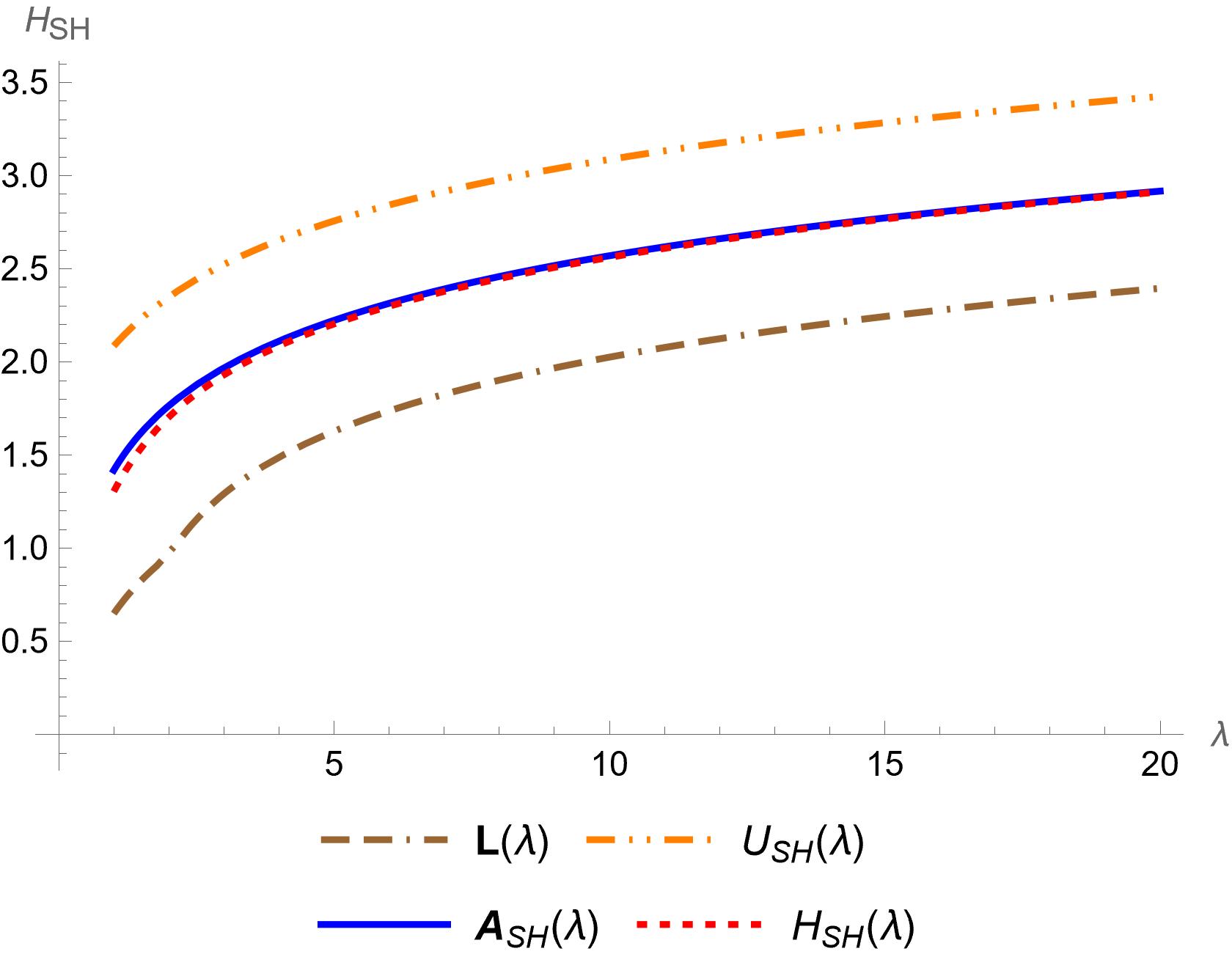}
\caption{A comparison of the Shannon entropy $\entropy_{SH}(\lambda)$ with the lower estimate $L(\lambda)$ given by \eqref{eq:lowerboundEnt}, the upper estimate $U_{SH}(\lambda)$ given by \eqref{eq:sh-bound}, and the asymptotic function $A_{SH}(\lambda)$ defined by the right hand side of \eqref{eq:asySH}.}
\label{fig:estimShannon}
\end{figure}

\begin{figure}
\centering 
\includegraphics[width=\textwidth]{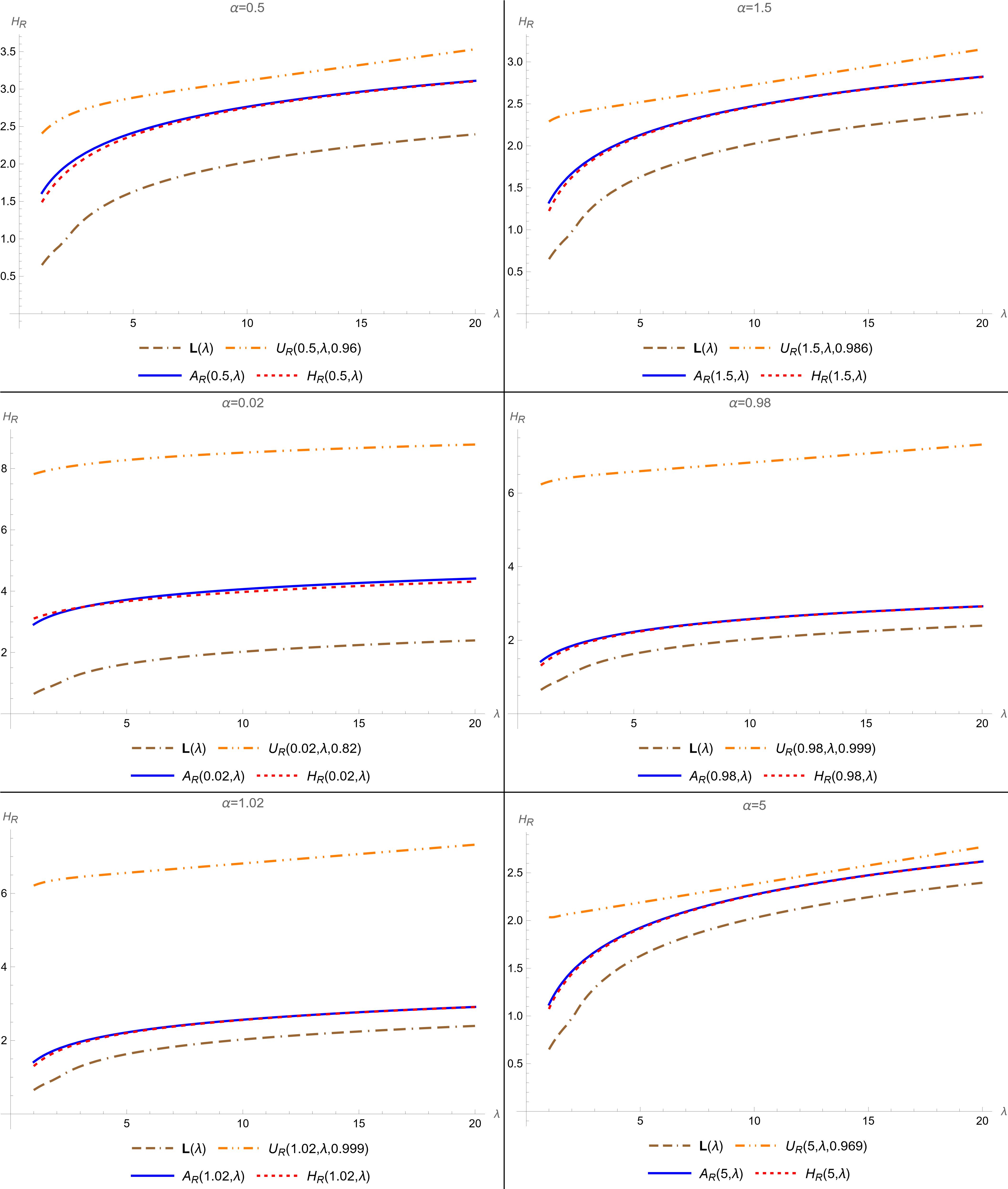}
\caption{A comparison of the R\'enyi entropy $\entropy_{R}(\alpha,\lambda)$ with the lower estimate $L(\lambda)$ given by \eqref{eq:lowerboundEnt}, the upper estimate $U_{R}(\alpha,\lambda,\gamma)$ given by \eqref{eq:estR-up1}--\eqref{eq:estR-up2} (with optimal values of $\gamma$ for $\lambda\in[1,20])$, and the asymptotic function $A_{R}(\alpha, \lambda)$ defined by the right hand side of \eqref{eq:asyR}.}
\label{fig:estimRenyi}
\end{figure}

\section{Monotonicity of entropies of Poisson distribution as
the function of intensity parameter $\lambda$}
\label{sec:monotonicity}

As it was explained in Introduction, it is reasonable to expect that the entropy of the Poisson distribution increases with intensity $\lambda$. It is natural to characterize such behavior as ``normal''. However, the most interesting cases are those where ``normal'' behavior is violated and entropies behave in strange ways. Such behavior can be naturally called ``anomalous''.

A summary of this Section is that Shannon, R\'enyi, Tsallis, and  Sharma--Mittal entropies always behave ``normally'', while both generalized R\'enyi entropies exhibit ``anomalous'' behavior for some values of their parameters.

\subsection{``Normal`` behavior}

\begin{proposition}\label{prop:normal}
    For the Poisson distribution with parameter $\lambda>0$: 
    \begin{enumerate}[(i)]
    \item {\normalfont\textrm{\cite[Theorem 1]{braiman2024}}} The Shannon entropy $\entropy_{SH}(\lambda)$ is strictly increasing and concave as a function of $\lambda\in(0,\infty)$. 
\item {\normalfont\textrm{\cite[Corollary 1]{braiman2024}}}
    For any $\alpha \in (0,1) \cup (1,\infty)$, the R\'enyi entropy $\entropy_R(\alpha,\lambda)$ strictly increases as a function of $\lambda\in(0,\infty)$.
    \item 
    For any $\alpha \in (0,1) \cup (1,\infty)$, the Tsallis entropy $\entropy_{T}(\alpha, \lambda)$ strictly increases as a function of~$\lambda\in(0,\infty)$.
    \item 
    For any $\alpha,\beta \in (0,1) \cup (1,\infty)$, the  Sharma--Mittal entropy $\entropy_{SM}(\alpha, \beta, \lambda)$ strictly increases as a function of $\lambda\in(0,\infty)$.
    \end{enumerate}
\end{proposition}

\begin{proof}
For $(i)$, see the proof of \cite[Theorem 1]{braiman2024}. 

Both $(ii)$ and $(iii)$ follow immediately from Proposition~\ref{prop:psi-monot}. Indeed, for $0<\alpha<1$, both functions $\log\psi(\alpha,\lambda)$ and $\psi(\alpha,\lambda)-1$ increase in $\lambda>0$ and since then $1-\alpha>0$, both entropies $\entropy_{R}$ and $\entropy_{TS}$ increase. Similarly, for $\alpha>1$, both these functions decrease in $\lambda>0$ but then in $1-\alpha<0$, so the corresponding entropies increase anyway.

To prove $(iv)$, we notice that, by Proposition~\ref{prop:psi-monot}, the function
$\lambda\mapsto (\psi(\alpha,\lambda))^{\frac{1}{1 - \alpha}}$ strictly increases in both cases $0 < \alpha < 1$ and $\alpha >1$. Now consider two cases:

\begin{itemize}
 \item[-] If $0 < \beta < 1$, then $(\psi(\alpha,\lambda))^{\frac{1 - \beta}{1 - \alpha}}$ increases  in $\lambda$,  and since $1-\beta>0$, the entropy $\entropy_{SM}(\alpha, \beta, \lambda)$  also increases in $\lambda$.
\item[-] If $\beta > 1$, then $(\psi(\alpha,\lambda))^{\frac{1 - \beta}{1 - \alpha}}$ decreases in $\lambda$, but then $1-\beta<0$, hence, the entropy $\entropy_{SM}(\alpha, \beta, \lambda)$  increases in $\lambda$ anyway. \qedhere
\end{itemize}
\end{proof}

\begin{remark}
Figure~\ref{fig:tsallis} illustrates the graph of Tsallis entropy as a function of variables $\alpha$ and $\lambda$. The graph confirms that the Tsallis entropy is a positive function, increasing with respect to $\lambda$ for any $\alpha > 0$ and $\alpha \neq 1$, which is consistent with our theoretical findings. Furthermore, the entropy decreases with respect to $\alpha$ for any $\lambda > 0$, aligning with theoretical expectations. This behavior is explained by the sum $\sum_{i = 0}^{\infty} p_i^\alpha$ in the definition \eqref{eq:T}, which decreases as $\alpha$ increases. Note also that for $\alpha>1$, the entropy is bounded by/increases to $\frac{1}{1-\alpha}$, see \eqref{eq:powergrowthT2}.

\begin{figure}
\centering
\begin{subfigure}{0.48\textwidth}
    \includegraphics[width=\textwidth]{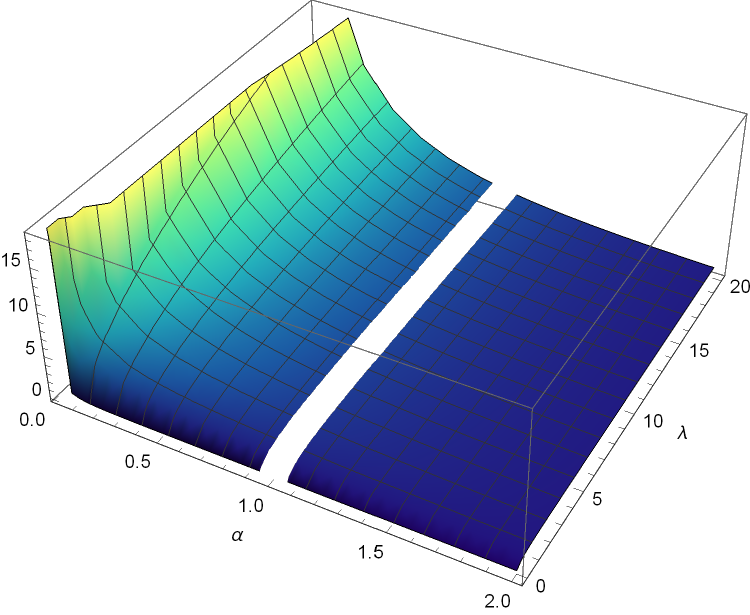}
    \caption{$\entropy_T(\alpha,\lambda)$ as a function of $\alpha$ and $\lambda$}
\end{subfigure}
\hfill
\begin{subfigure}{0.48\textwidth}
    \includegraphics[width=\textwidth]{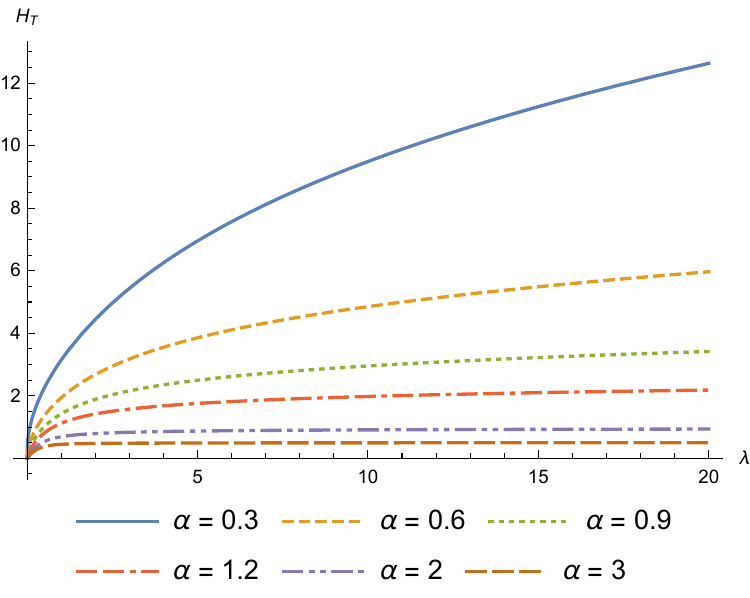}
    \caption{$\entropy_T(\alpha,\lambda)$ as a function of~$\lambda$ for various fixed $\alpha$}
\end{subfigure}
\caption{The Tsallis entropy $\entropy_T(\alpha,\lambda)$}
\label{fig:tsallis}
\end{figure}
\end{remark}

\begin{remark}
Figures~\ref{fig:sm-alpha=2}--\ref{fig:sm-beta=2} present the behavior of Sharma--Mittal entropy for the Poisson distribution under various parameter settings ($\alpha > 0$, $\beta > 0$, $\alpha\neq\beta$). In all presented cases, the entropy $\entropy_{SM}(\alpha, \beta, \lambda)$ increases as a function of $\lambda$. 

Furthermore, we can see that the Sharma--Mittal entropy decreases with respect to $\alpha$ for any fixed $\beta$ and $\lambda$, and similarly decreases with respect to $\beta$ for any fixed $\alpha$ and $\lambda$ in line with the theoretical prediction. Indeed, such monotonic behavior in $\alpha$ and $\beta$ can be generalized to any discrete distribution, not just the Poisson distribution, by considering the sum $\sum_{i = 0}^{\infty} p_i^\alpha$ in the definition \eqref{eq:SM}, which decreases with increasing $\alpha$.

Finally, on Figure~\ref{fig:sm-alpha=2}, we can see that the functions  $\entropy_{SM}(2, \beta, \lambda)$ grow to different finite values when $\beta>1$ (specifically, to $\frac1{\beta-1}$, see \eqref{eq:powergrowthSM2}). In contrast, on Figure~\ref{fig:sm-beta=2}, we can see that all the functions $\entropy_{SM}(\alpha, 2, \lambda)$ have the same limiting behavior (and apparently converge to $1=\frac1{\beta-1}\bigr\rvert_{\beta=2}$) for different $\alpha$. This is not the case when $\beta\in(0,1)$ (see Theorem~\ref{thm:limits}(c)) as then $\entropy_{SM}(\alpha, \beta, \lambda)$ grows as $c(\alpha,\beta)\lambda^{\frac{1-\beta}{2}}$ for all $\alpha$ that is reflected in Figure~\ref{fig:sm-beta=0.5}.
\end{remark}

\begin{figure}
\centering
    \begin{subfigure}{0.48\textwidth}
        \includegraphics[width=\textwidth]{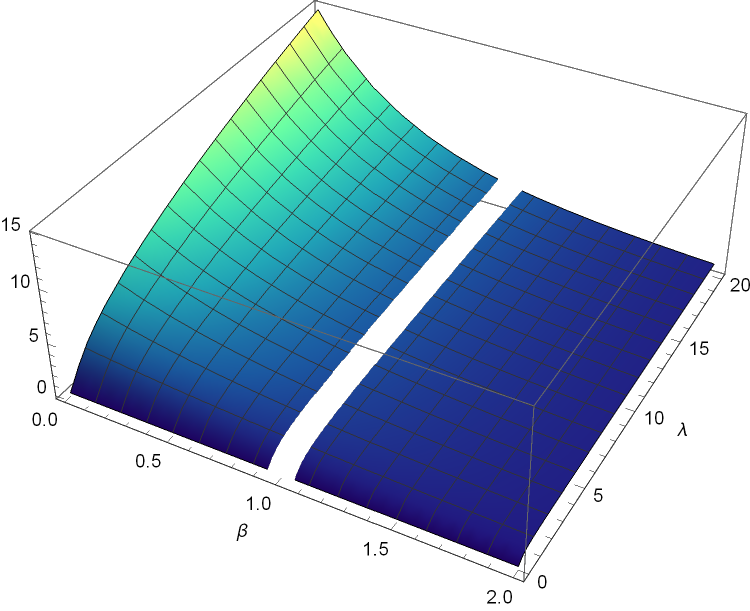}
        \caption{$\entropy_{SM}(2, \beta, \lambda)$ as a function of $\beta$ and $\lambda$}
    \end{subfigure}
    \hfill
    \begin{subfigure}{0.47\textwidth}
        \includegraphics[width=\textwidth]{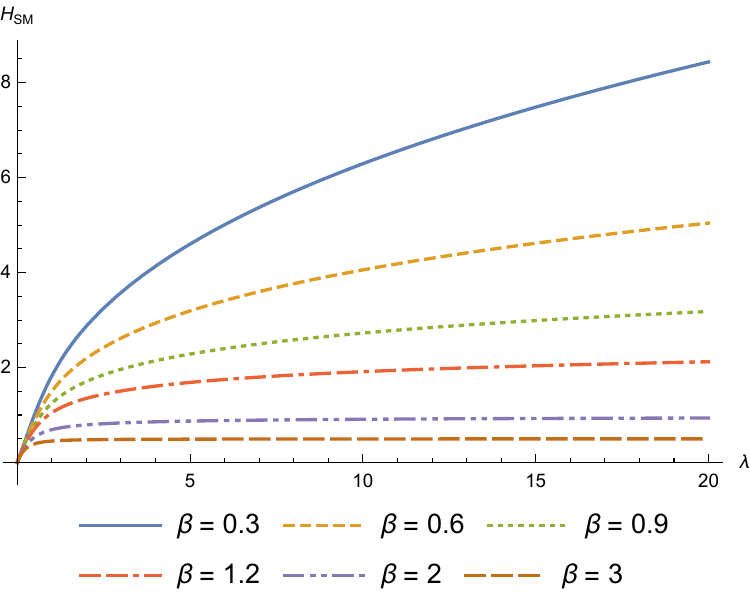}
        \caption{$\entropy_{SM}(2, \beta, \lambda)$ as a function $\lambda$ for various $\beta$}
    \end{subfigure}
\caption{The Sharma--Mittal entropy for $\alpha = 2$}
\label{fig:sm-alpha=2}
\end{figure}

\begin{figure}
\centering
    \begin{subfigure}{0.48\textwidth}
        \includegraphics[width=\textwidth]{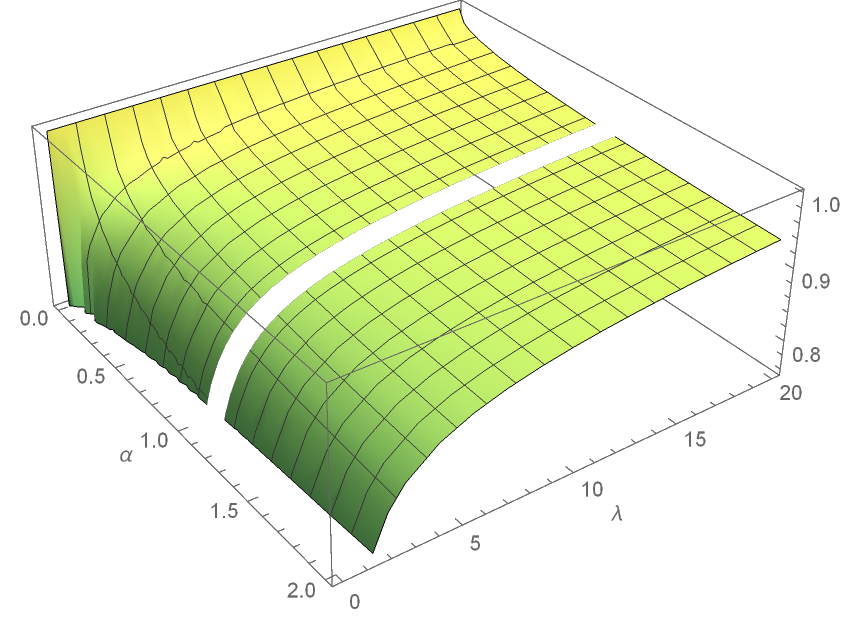}
        \caption{$\entropy_{SM}(\alpha, 2, \lambda)$ as a function of $\alpha$ and $\lambda$}
    \end{subfigure}
    \hfill
    \begin{subfigure}{0.48\textwidth}
        \includegraphics[width=\textwidth]{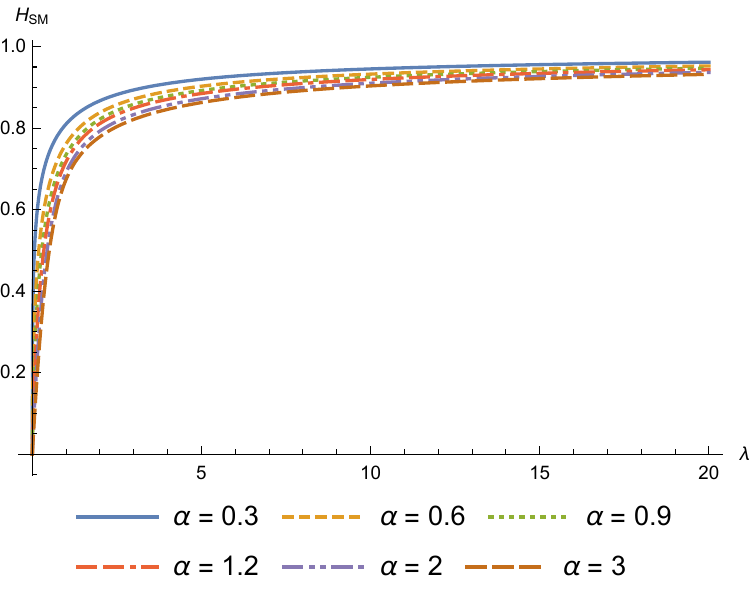}
        \caption{$\entropy_{SM}(\alpha, 2, \lambda)$ as a function $\lambda$  for various $\alpha$}
    \end{subfigure}
    \caption{The Sharma--Mittal entropy for $\beta = 2$}
    \label{fig:sm-beta=2}
\end{figure}

\begin{figure}
    \centering
    \begin{subfigure}{0.48\textwidth}
        \includegraphics[width=\textwidth]{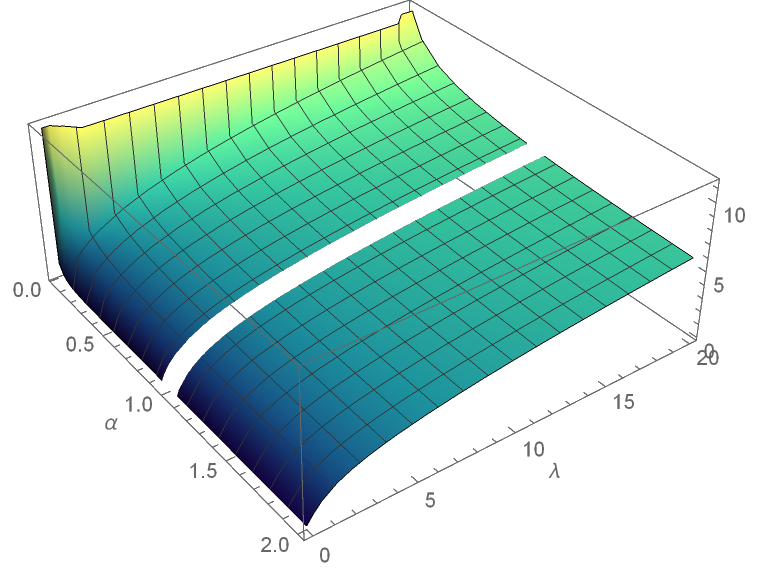}
        \caption{$\entropy_{SM}(\alpha, 0.5, \lambda)$ as a function of $\alpha$ and $\lambda$}
    \end{subfigure}
    \hfill
    \begin{subfigure}{0.48\textwidth}
        \includegraphics[width=\textwidth]{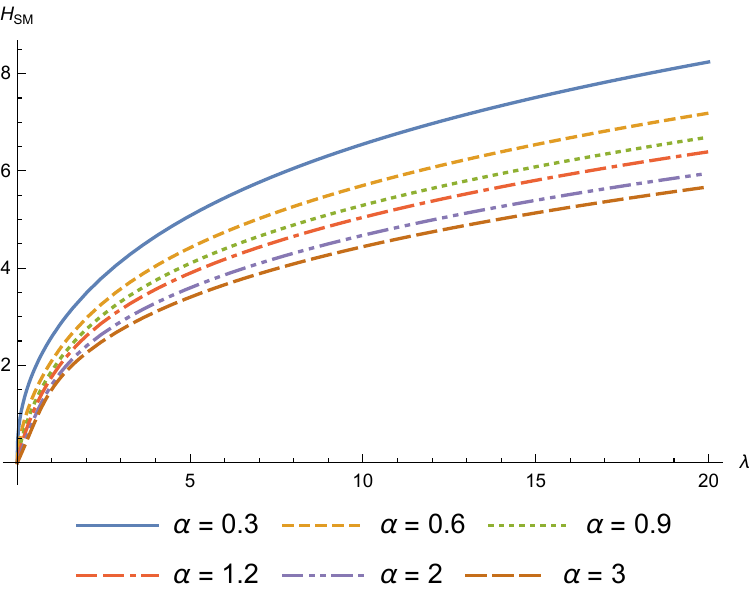}
        \caption{$\entropy_{SM}(\alpha, 0.5, \lambda)$ as a function $\lambda$  for various $\alpha$}
    \end{subfigure}
    \caption{The Sharma--Mittal entropy for $\beta = 0.5$}
    \label{fig:sm-beta=0.5}
\end{figure}

\subsection{``Anomalous`` behavior of generalized R\'enyi entropies}
\label{ssec:GR-anomalous}

We now turn to the study of two generalized R\'enyi entropies, which exhibit ``anomalous'' behavior for certain values of the parameters.

\subsubsection{The generalized R\'enyi entropy $\entropy_{GR}(\alpha,\lambda)$}
Let us first consider the generalized R\'enyi entropy with one parameter $\alpha > 0$.  By \eqref{eq:SHhtroughGR} and Proposition~\ref{prop:normal} (i), $\frac{\partial}{\partial \lambda}\entropy_{GR}(1,\lambda) =\frac{\partial}{\partial \lambda}\entropy_{SH}(\lambda)>0$ for all $\lambda>0$. The next statement shows that the positivity of $\frac{\partial}{\partial \lambda}\entropy_{GR}(\alpha,\lambda)$ may fail for $\alpha$ in a neighborhood of $0$.

\begin{proposition}\label{prop:anomalGR1zero}
There exists an interval $J\subset(0,1)$ (in a neighborhood of $0$) such that, for each $\alpha \in J$, $\entropy_{GR}(\alpha, \lambda)$ is decreasing as a function of $\lambda$ in the neighborhood of $\lambda = 1$.
\end{proposition}

\begin{proof}
By \eqref{eq:entGR-psi},
\begin{equation}\label{eq:derder}
    \frac{\partial}{\partial \lambda} \entropy_{GR}(\alpha,\lambda) 
= - \frac{\partial^2}{\partial\alpha\partial\lambda} \log \psi(\alpha, \lambda).
\end{equation}
We are going to prove that, for some $J\subset(0,1)$,
\begin{equation}\label{eq:shallprove}
    \frac{\partial}{\partial \lambda} \entropy_{GR}(\alpha,\lambda) \Bigr\rvert_{\lambda = 1}<0, \quad \alpha\in J.
\end{equation}
By the continuity of $\frac{\partial^2}{\partial\alpha\partial\lambda} \log \psi(\alpha, \lambda)$ (see Remark~\ref{rem:mixed-der}), \eqref{eq:shallprove} would imply that, for each such $\alpha\in J$, $\frac{\partial}{\partial \lambda} \entropy_{GR}(\alpha,\lambda)$ for all $\lambda$ from a neighborhood of $1$ (the neighborhood may depend on $\alpha$) that would prove the statement. 

We define
\[
\rho(\alpha) := \frac{\partial}{\partial \lambda}\log\psi(\alpha,\lambda)\Bigr\rvert_{\lambda = 1},
\]
so that, by \eqref{eq:derder},
\begin{equation}
    \frac{\partial}{\partial \lambda} \entropy_{GR}(\alpha,\lambda) \Bigr\rvert_{\lambda = 1} = - \rho'(\alpha).\label{eq:keyrel}
\end{equation}
Using formula \eqref{eq:der-psi-lambda}, we get
\[
\frac{\partial}{\partial \lambda}\log\psi(\alpha,\lambda) 
= \frac{\frac{\partial}{\partial \lambda}\psi(\alpha,\lambda) }{\psi(\alpha,\lambda)}
= \frac{\alpha \sum_{i=0}^\infty (i-\lambda) \frac{\lambda^{\alpha i - 1}}{(i!)^\alpha}}{\sum_{i = 0}^{\infty} \frac{\lambda^{i \alpha} }{(i!)^\alpha}},
\]
and therefore, 
\[
\rho(\alpha) =
\alpha \left(\frac{\sum_{i=1}^\infty \frac{i}{(i!)^\alpha}}{\sum_{i=0}^\infty \frac{1}{(i!)^\alpha}}-1\right).
\]

Note that by Proposition \ref{prop:psi-monot} $(i)$, $\log\psi(\alpha,\lambda)$ strictly increases as a function of $\lambda$ for any $\alpha\in(0,1)$. Then $ \rho(\alpha) > 0$ for any $\alpha\in(0,1)$, and for $\alpha =1$ we have
\[
\rho(1)
= \left(\frac{\sum_{i=1}^\infty \frac{1}{(i-1)!}}{\sum_{i=0}^\infty \frac{1}{i!}}-1\right)
= 0.
\]
Moreover, by to Lemma \ref{l:lim-ratio}, $\rho(\alpha)\to0$ as $\alpha\to0$. Note also that both $\rho(\alpha)$ and $\rho'(\alpha)$ are continuous functions (by Lemma~\ref{l:der-psi} and Remark \ref{rem:mixed-der}). Therefore, $\rho$ is a strictly positive continuously differentiable function on $(0,1)$ which converges to $0$ as $\alpha\to0$. As a result, $\rho$ must be increasing at a point, and hence, on an interval $J$, in a neighborhood of $0$. Then $\rho'(\alpha)>0$ for $\alpha\in J$, that, by \eqref{eq:keyrel}, proves \eqref{eq:shallprove}. The statement is proven.
\end{proof}

\begin{remark}
The graphs of $\rho(\alpha) \coloneqq \frac{\partial}{\partial \lambda} \log \psi(\alpha, \lambda)\rvert_{\lambda = 1}$ and $\rho'(\alpha) = \frac{\partial^2}{\partial \alpha \partial \lambda} \log \psi(\alpha, \lambda)\rvert_{\lambda = 1}$ for $\alpha \in (0,1)$ are shown in Figures \ref{fig:rho} (a) and (b), respectively. Figures \ref{fig:rho} (c) and (d) present the graphs of $\rho(\alpha)$ and $\rho'(\alpha)$ on an expanded scale near zero. It can be observed that there is indeed an interval where $\rho(\alpha)$ increases, as it was proven in Proposition~\ref{prop:anomalGR1zero}. Additionally, the graphs suggest that this interval is of the form $(0, \alpha_0)$, with $\alpha_0 \approx 0.15$. Beyond this point, $\rho(\alpha)$ decreases for $\alpha \in (\alpha_0, 1]$, eventually reaching $\rho(1) = 0$.
\end{remark}

\begin{figure}[htb]
    \centering
    \begin{subfigure}{0.43\textwidth}
        \includegraphics[width=\textwidth]{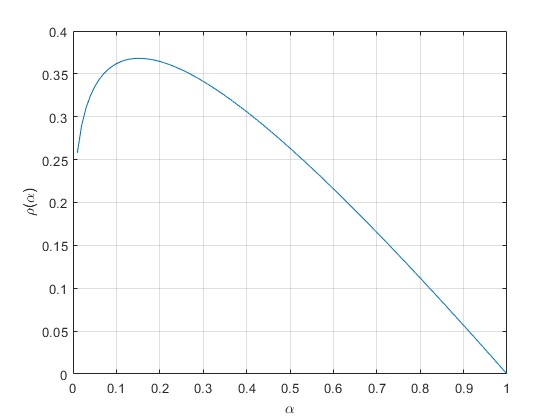}
        \caption{$\rho(\alpha)$ for $\alpha\in(0,1)$}
    \end{subfigure}
    \hfill
    \begin{subfigure}{0.43\textwidth}
        \includegraphics[width=\textwidth]{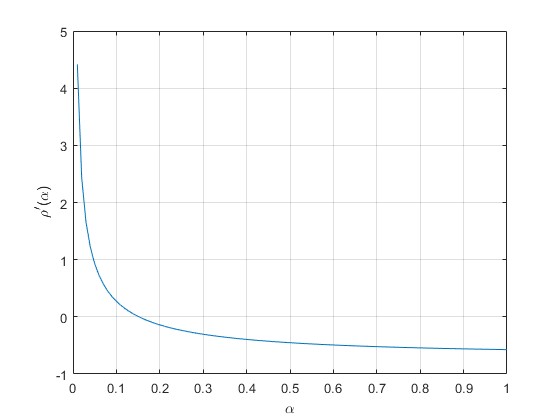}
        \caption{$\rho'(\alpha)$ for $\alpha\in(0,1)$}
    \end{subfigure}
    \hfill
    \begin{subfigure}{0.43\textwidth}
        \includegraphics[width=\textwidth]{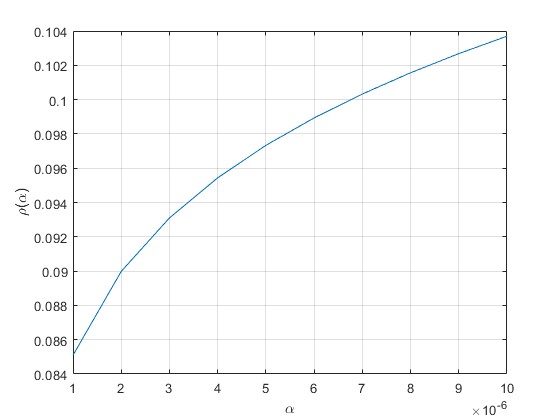}
        \caption{$\rho(\alpha)$ for $\alpha\in(10^{-6},10^{-5})$}
    \end{subfigure}
    \hfill
    \begin{subfigure}{0.43\textwidth}
        \includegraphics[width=\textwidth]{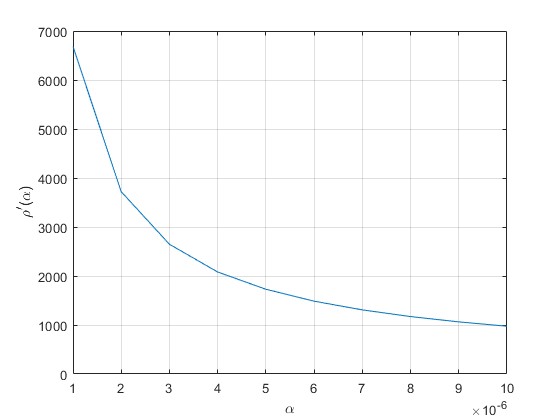}
        \caption{$\rho'(\alpha)$ for $\alpha\in(10^{-6},10^{-5})$}
    \end{subfigure}
    \caption{Graphs of the function $\rho(\alpha)$ and its derivative $\rho'(\alpha)$.}
    \label{fig:rho}
\end{figure}

\begin{remark}
   Figure \ref{subfig:GR1-full} displays the surface of the generalized R\'enyi entropy $\entropy_{GR}(\alpha, \lambda)$ within the region $0.05 \le \alpha \le 0.15$ and $0 < \lambda \le 20$. Figures \ref{subfig:GR1-a01} and \ref{subfig:GR1-a014} present the graphs of $\entropy_{GR}(\alpha, \lambda)$ as a function of $\lambda$ for $\alpha = 0.1$ and $\alpha = 0.14$, respectively. For these small values of $\alpha$, the entropy exhibits what we term as ``anomalous'' behavior, characterized by a decrease in $\lambda$ over an interval. After reaching a local minimum at the end of this interval, the entropy  $\entropy_{GR}(\alpha, \lambda)$ subsequently increases.
\end{remark}
\begin{figure}[htb]
    \centering
    \begin{subfigure}{0.48\textwidth}
        \includegraphics[width=\textwidth]{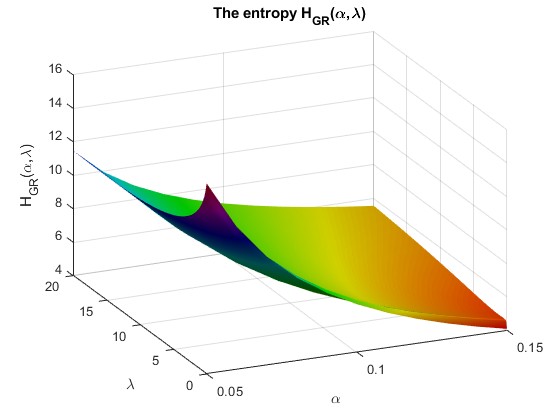}
        \caption{}
        \label{subfig:GR1-full}
    \end{subfigure}
    \hfill
    \begin{subfigure}{0.48\textwidth}
        \includegraphics[width=\textwidth]{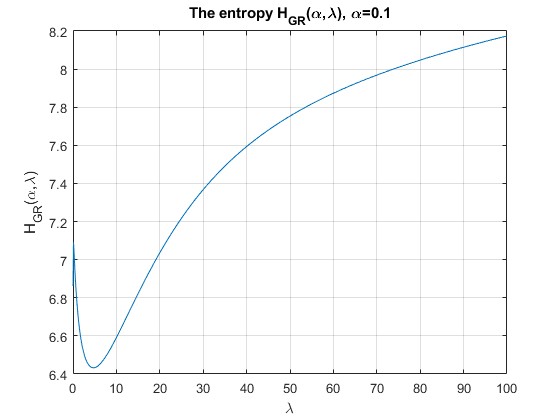}
        \caption{}
        \label{subfig:GR1-a01}
    \end{subfigure}
    \hfill
    \begin{subfigure}{0.48\textwidth}
        \includegraphics[width=\textwidth]{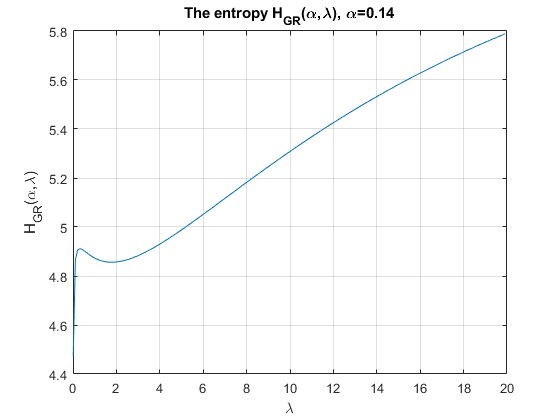}
        \caption{}
        \label{subfig:GR1-a014}
    \end{subfigure}
    \caption{The generalized R\'enyi entropy $\entropy_{GR}(\alpha,\lambda)$: (a)~as a function of $(\alpha,\lambda)$, (b)~as a function of $\lambda$ for $\alpha = 0.1$, (b)~as a function of $\lambda$ for $\alpha = 0.14$}
    \label{fig:GRplots}
\end{figure}

\begin{remark} \label{rem:anomGR1largealpha}
    We can also observe numerically that the generalized R\'enyi entropy $\entropy_{GR}(\alpha, \lambda)$ demonstrates ``anomalous'' behavior for large values of $\alpha$ as well, see Figure~\ref{fig:HGR1-largealphas-all}. 
    
    Namely, on Figure~\ref{subfig:HGR1-largealphas}, we can see that (for $\alpha=25,50,100$) the steepness of $\entropy_{GR}(\alpha, \lambda)$ changes close to the integer values of $\lambda$. Zooming the graphs, one can see, nevertheless, that $\entropy_{GR}(25, \lambda)$ is indeed non-monotone in $\lambda\in[1,1.15]$, however, it is increasing in $\lambda\in[2,2.5]$, see Figure~\ref{subfig:HGR1-alpha=25-zoom}. Next,
    $\entropy_{GR}(50, \lambda)$ is non-monotone in $\lambda\in[1,1.15]$ and $\lambda\in[2,2.15]$, however, it is increasing in $\lambda\in[3.05,3.25]$, see Figure~\ref{subfig:HGR1-alpha=50-zoom}. Finally, $\entropy_{GR}(100, \lambda)$ is non-monotone in $\lambda\in[1,1.15]$, $\lambda\in[2,2.15]$, and $\lambda\in[3,3.15]$, see Figure~\ref{subfig:HGR1-alpha=100-zoom}.
    
Such irregular behavior of $\entropy_{GR}(\alpha, \lambda)$ is evidently determined by the behavior of $\frac{\partial}{\partial\lambda}\entropy_{GR}(\alpha, \lambda)$. Figure~\ref{fig:derHGR1} shows that the derivative has damping oscillations in $\lambda$ starting from $1$ for $\lambda=0$. However, the amplitude of the oscillations, for e.g. $\alpha=15$, is not large enough to reach $0$, i.e. 
$\frac{\partial}{\partial\lambda}\entropy_{GR}(15, \lambda)>0$ for all (observed) $\lambda$. When $\alpha$ grows, the amplitude increases, and hence the number of roots to $\frac{\partial}{\partial\lambda}\entropy_{GR}(\alpha, \lambda)=0$ increases.
    
\end{remark}
    \begin{figure}[!ht]
    \centering
    
    \begin{subfigure}{\textwidth}
    \centering
        \includegraphics[width=0.3\textwidth]{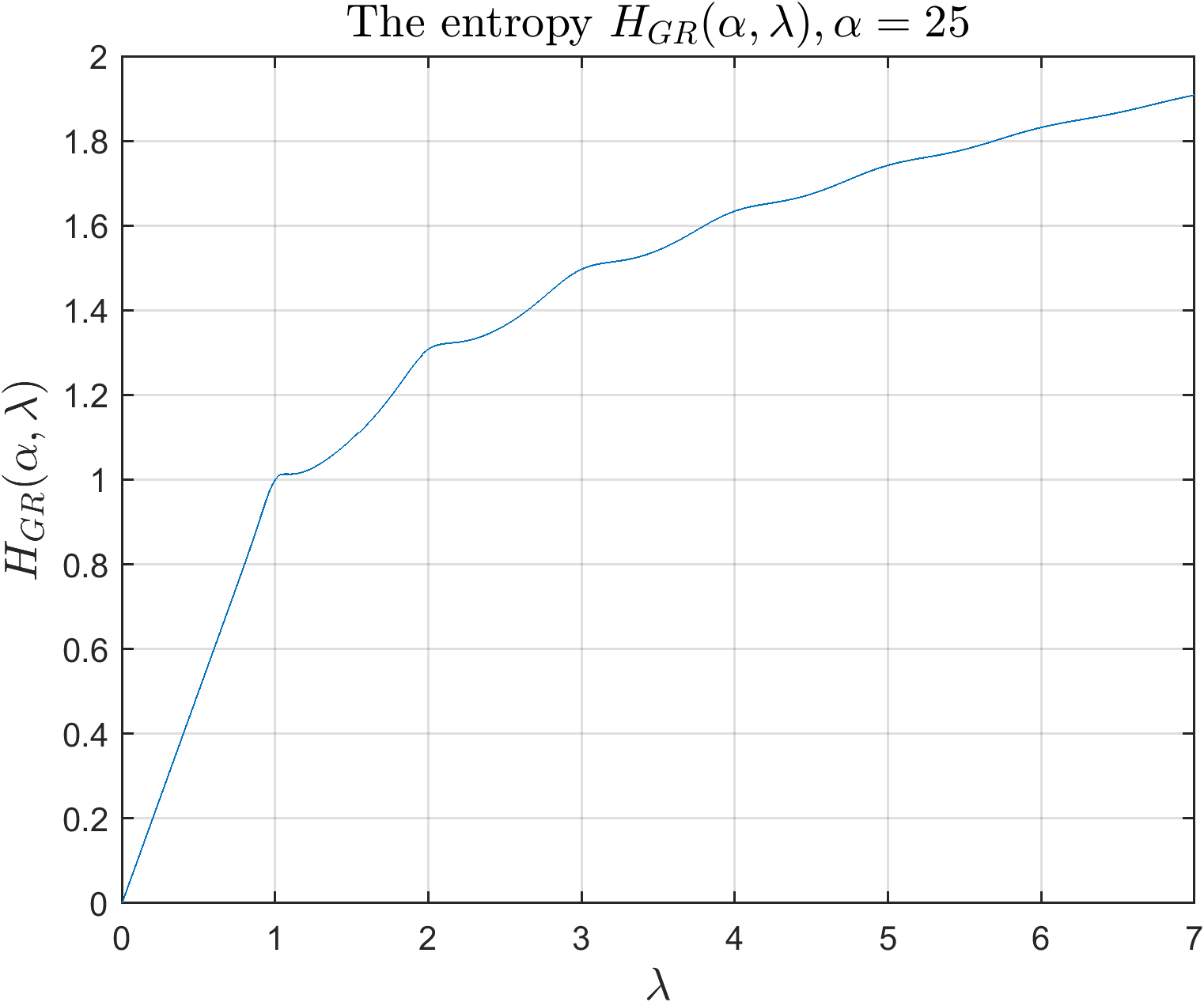}
        \includegraphics[width=0.3\textwidth]{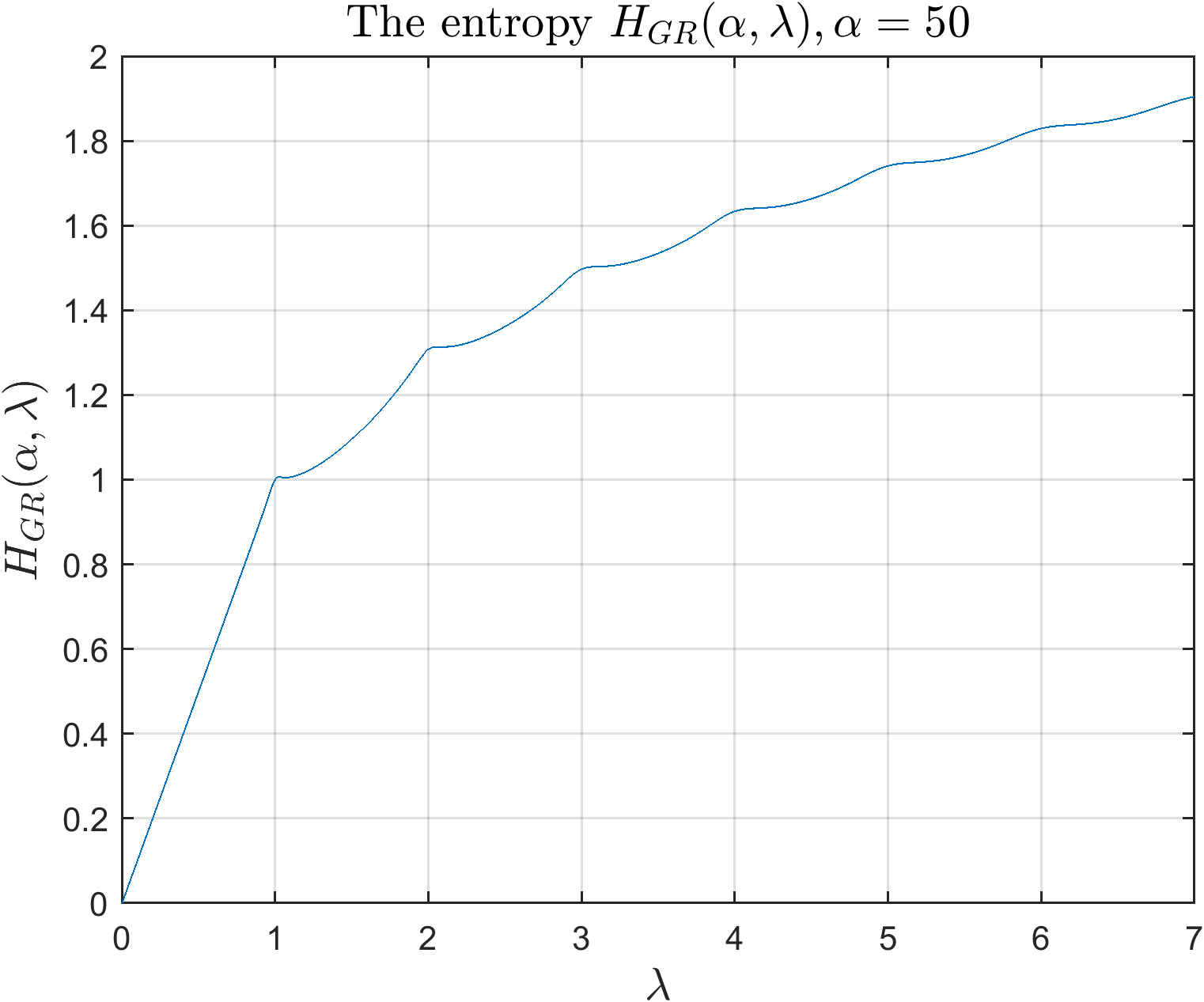}
        \includegraphics[width=0.3\textwidth]{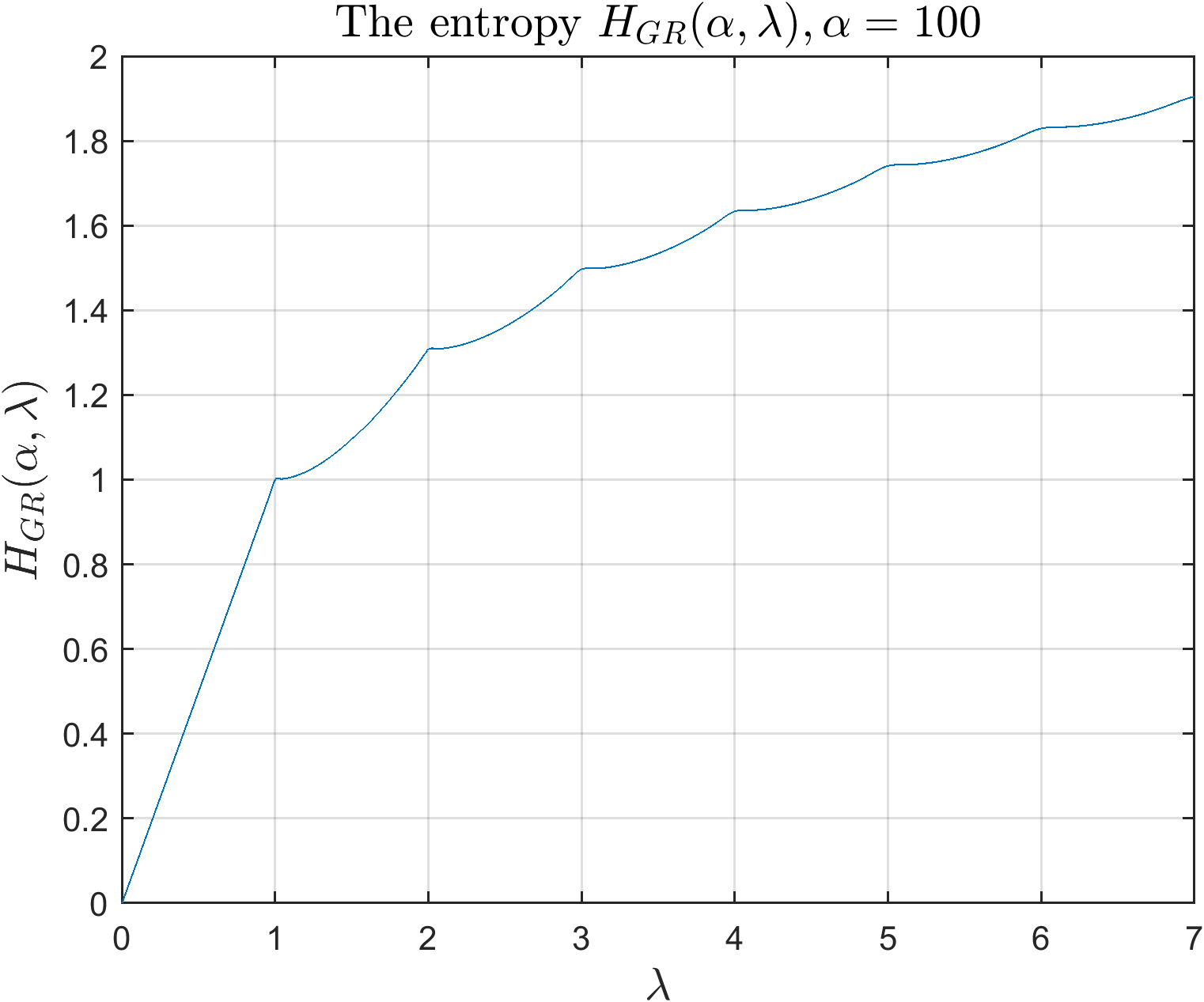}
    \caption{Non-monotone behavior of $\entropy_{GR}(\alpha,\lambda)$ for large values of $\alpha$}
    \label{subfig:HGR1-largealphas}
    \end{subfigure}

\bigskip

    \begin{subfigure}{\textwidth}
    \centering
        \includegraphics[width=0.3\textwidth]{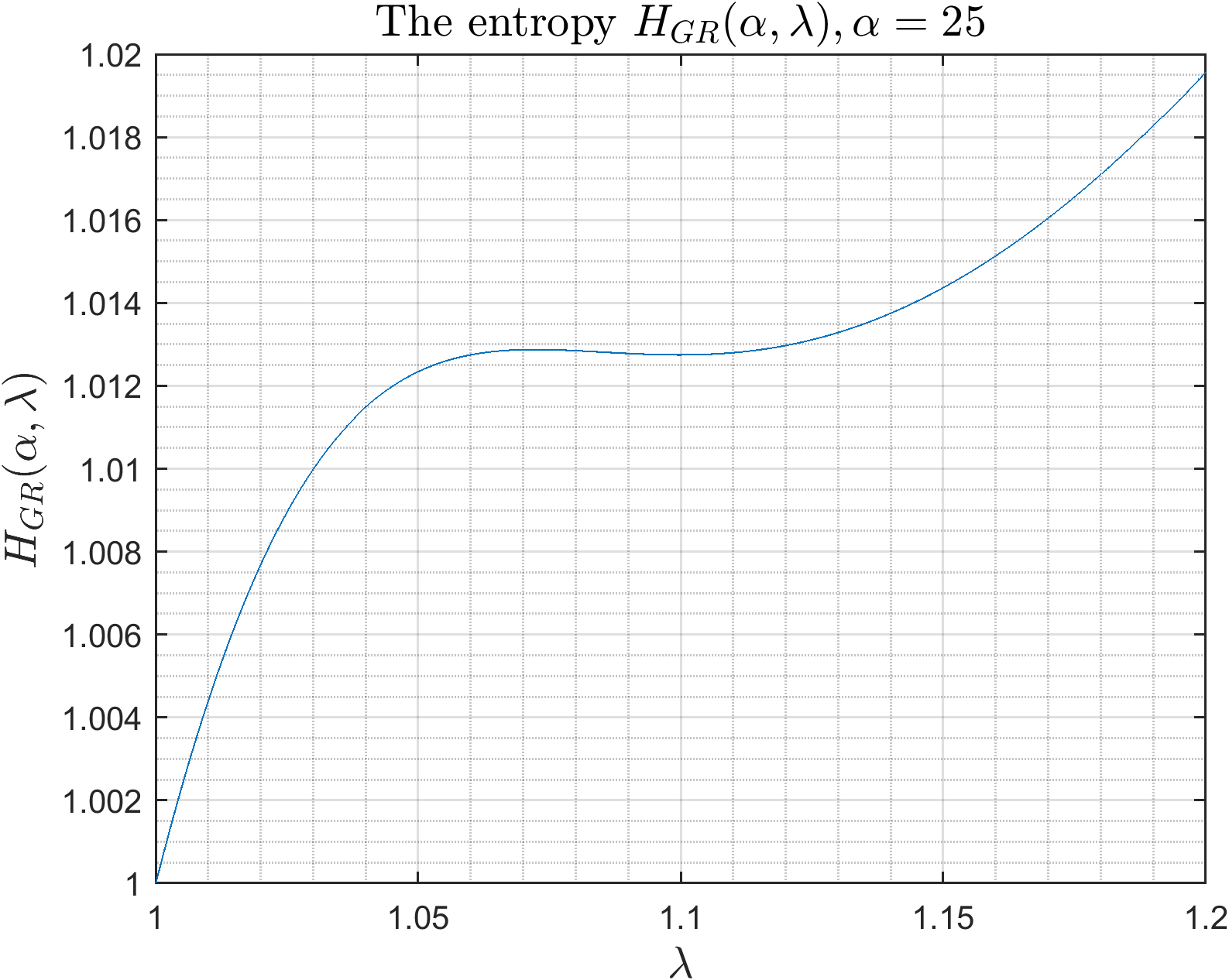}
        \includegraphics[width=0.3\textwidth]{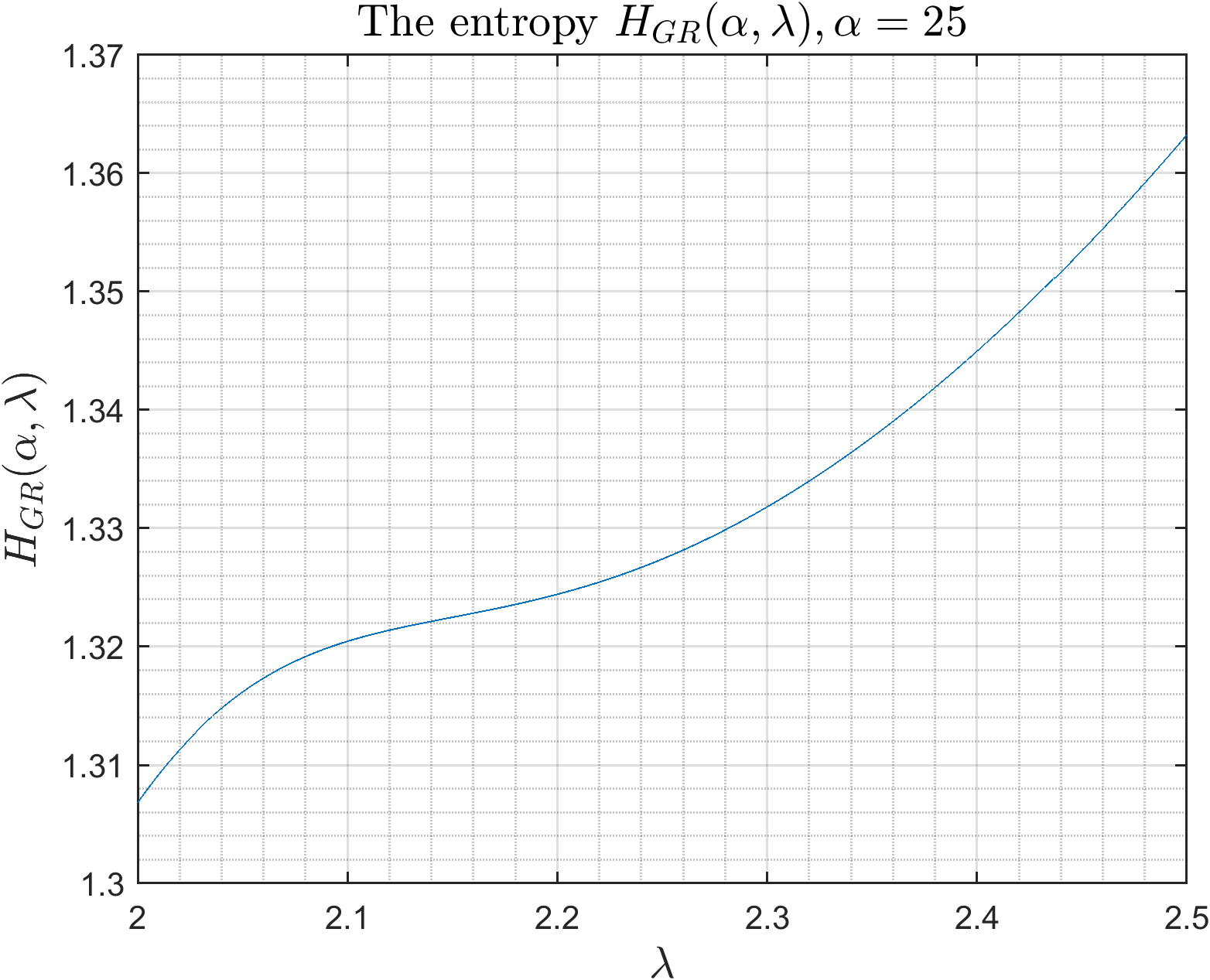}
    \caption{For $\alpha=25$, the monotonicity fails near $\lambda=1$, but it is preserved near $\lambda=2$}
    \label{subfig:HGR1-alpha=25-zoom}
    \end{subfigure}

\bigskip

    \begin{subfigure}{\textwidth}
    \centering
        \includegraphics[width=0.3\textwidth]{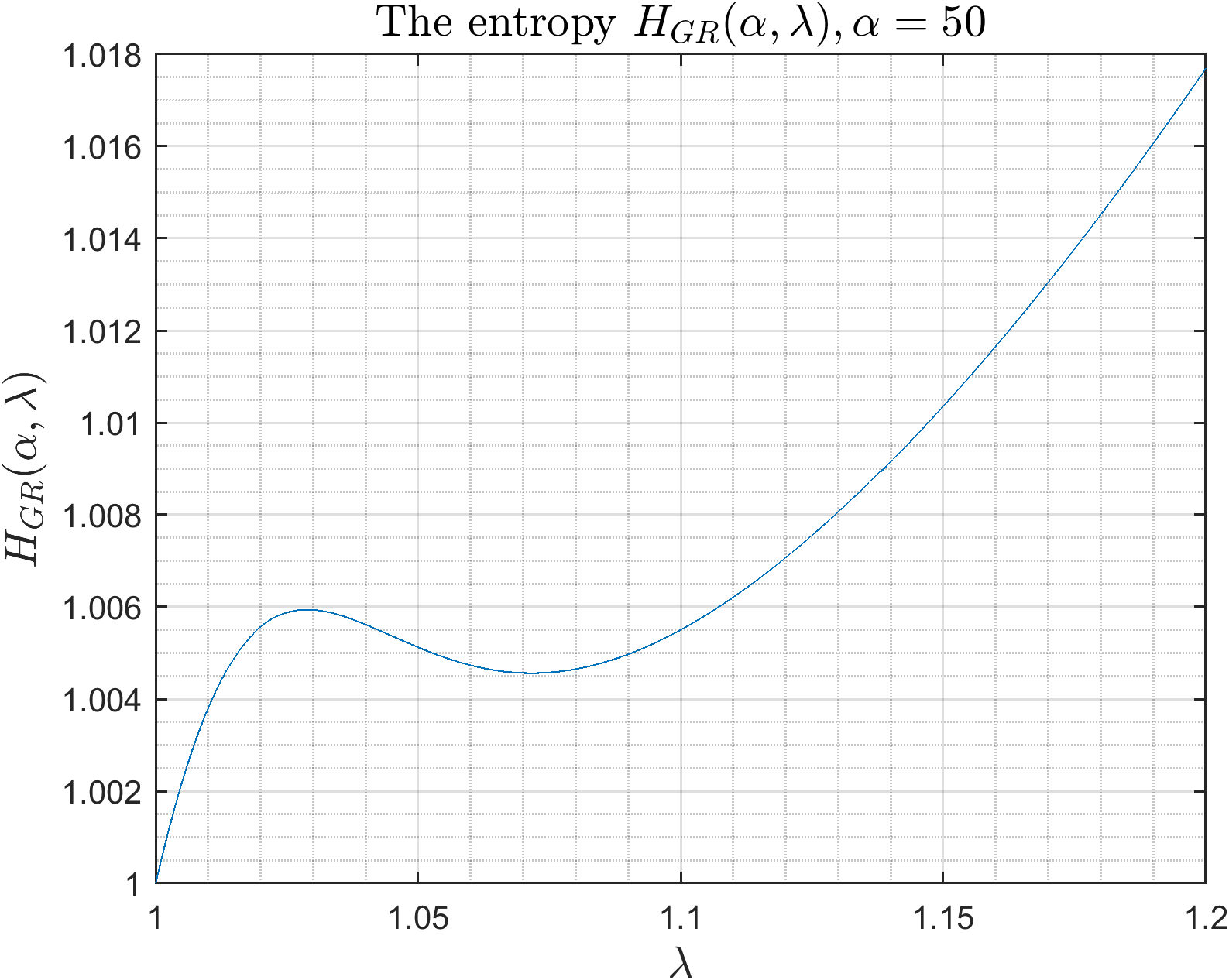}
        \includegraphics[width=0.3\textwidth]{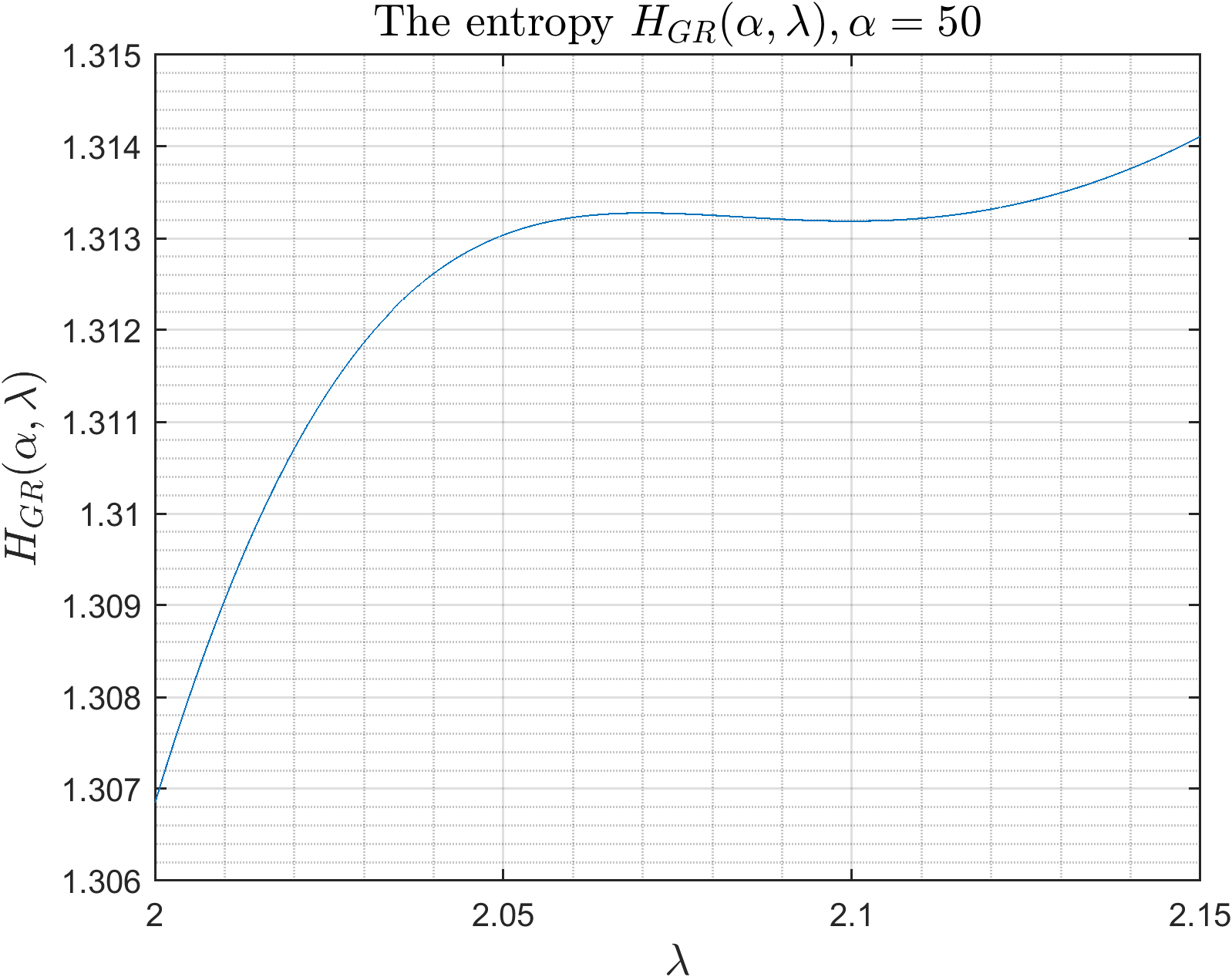}
        \includegraphics[width=0.3\textwidth]{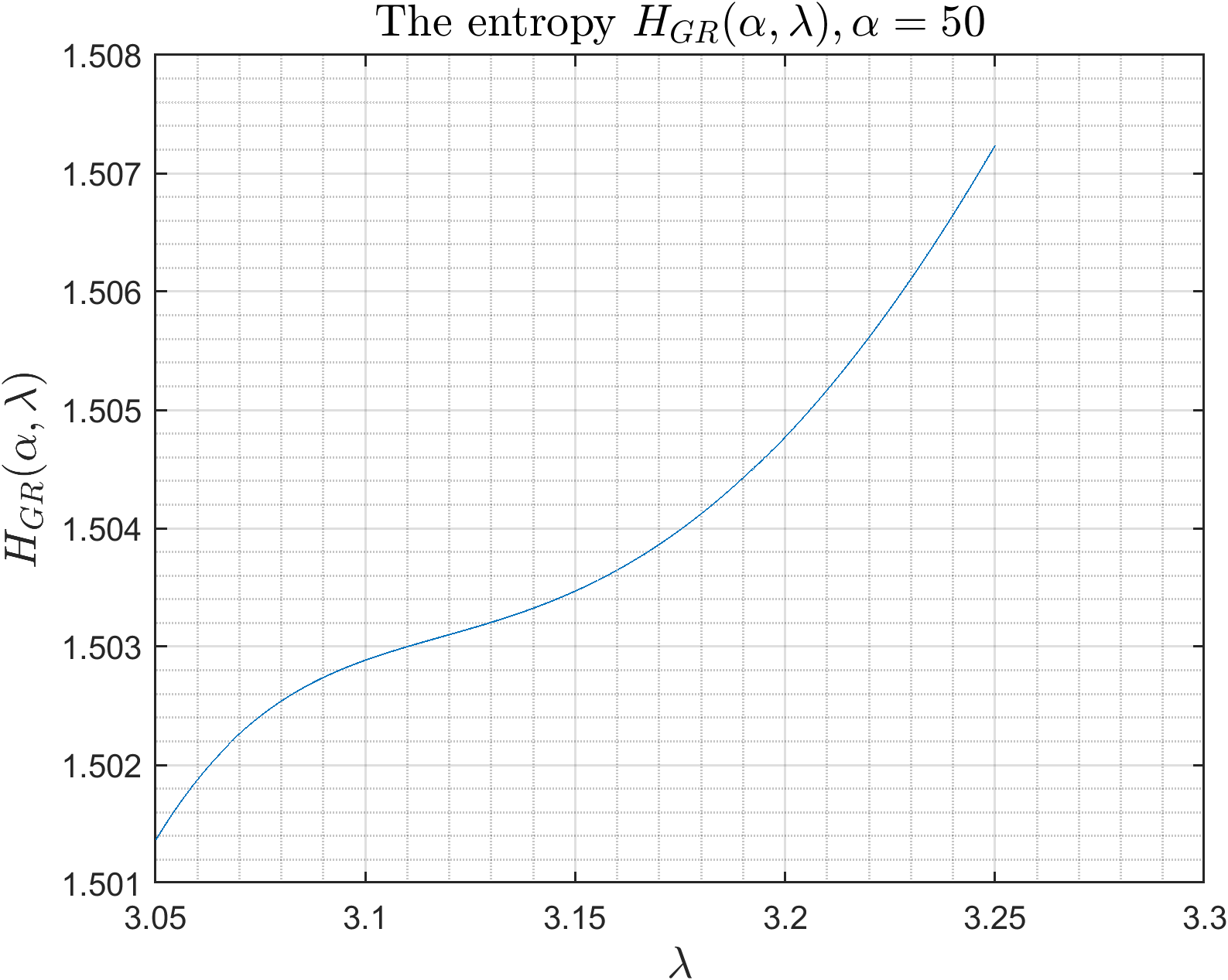}
    \caption{For $\alpha=50$, the monotonicity fails near $\lambda=1$ and $\alpha=2$, but it is preserved near $\lambda=3$}
    \label{subfig:HGR1-alpha=50-zoom}
    \end{subfigure}
   
\bigskip
 
    \begin{subfigure}{\textwidth}
    \centering
        \includegraphics[width=0.3\textwidth]{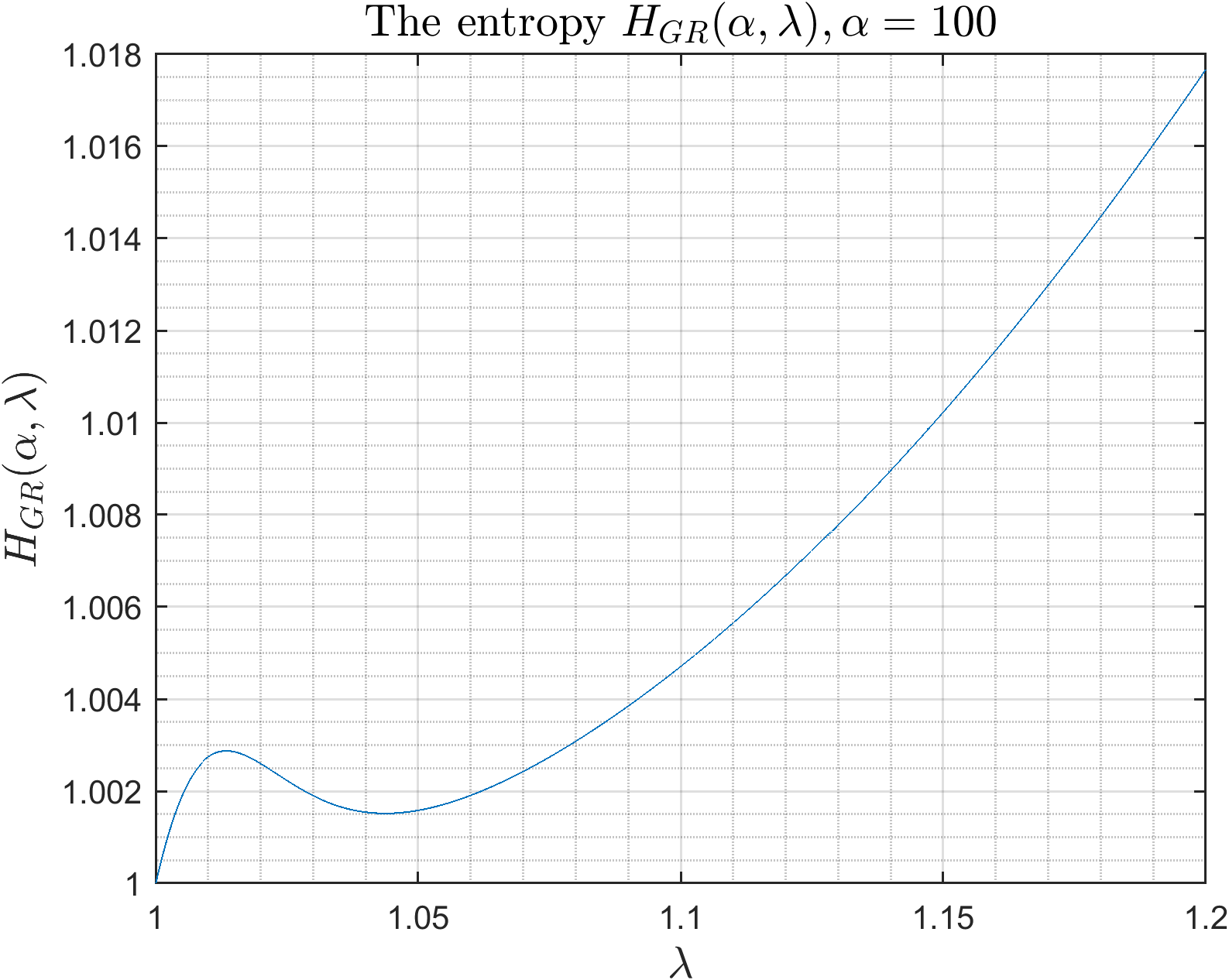}
        \includegraphics[width=0.3\textwidth]{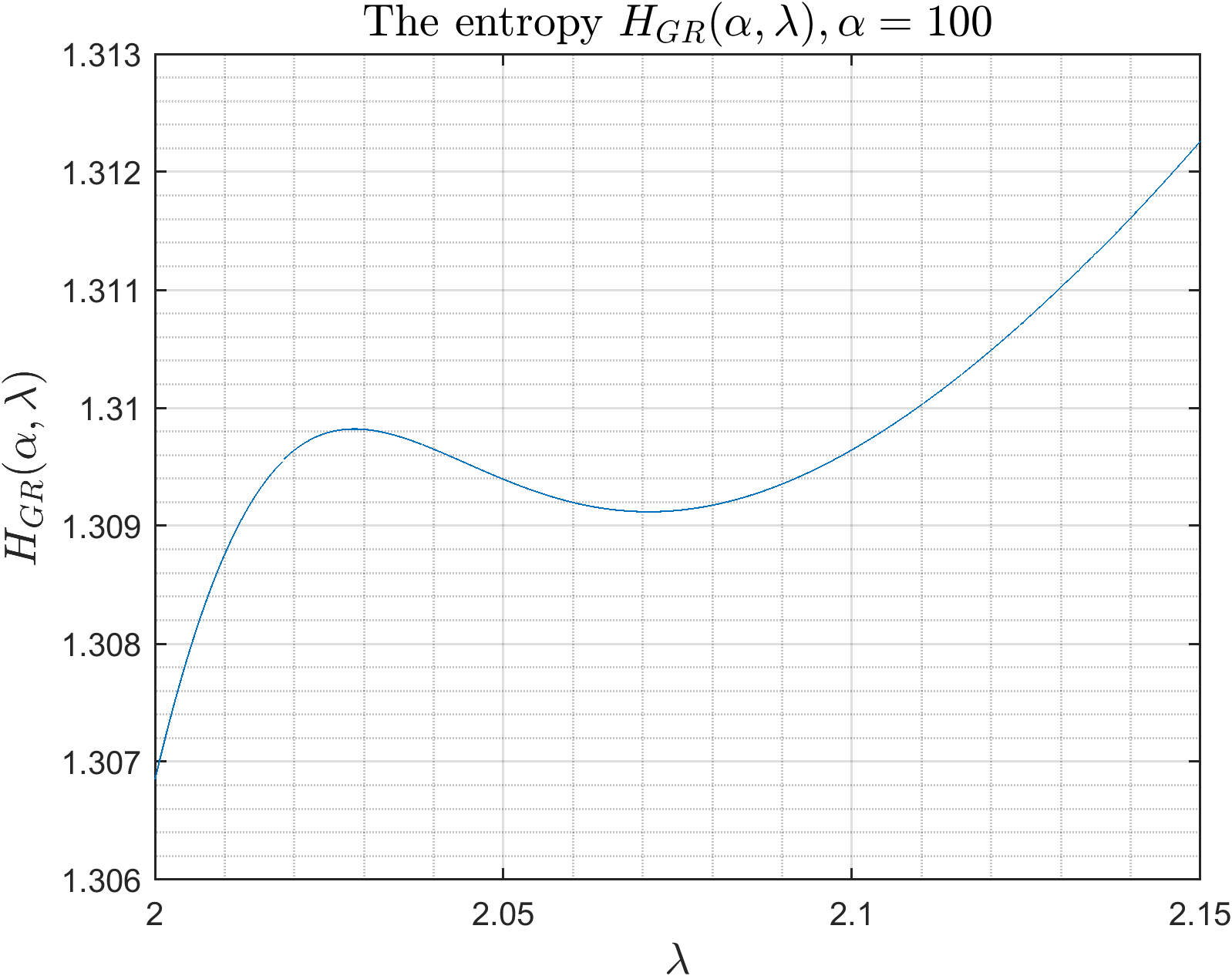}
        \includegraphics[width=0.3\textwidth]{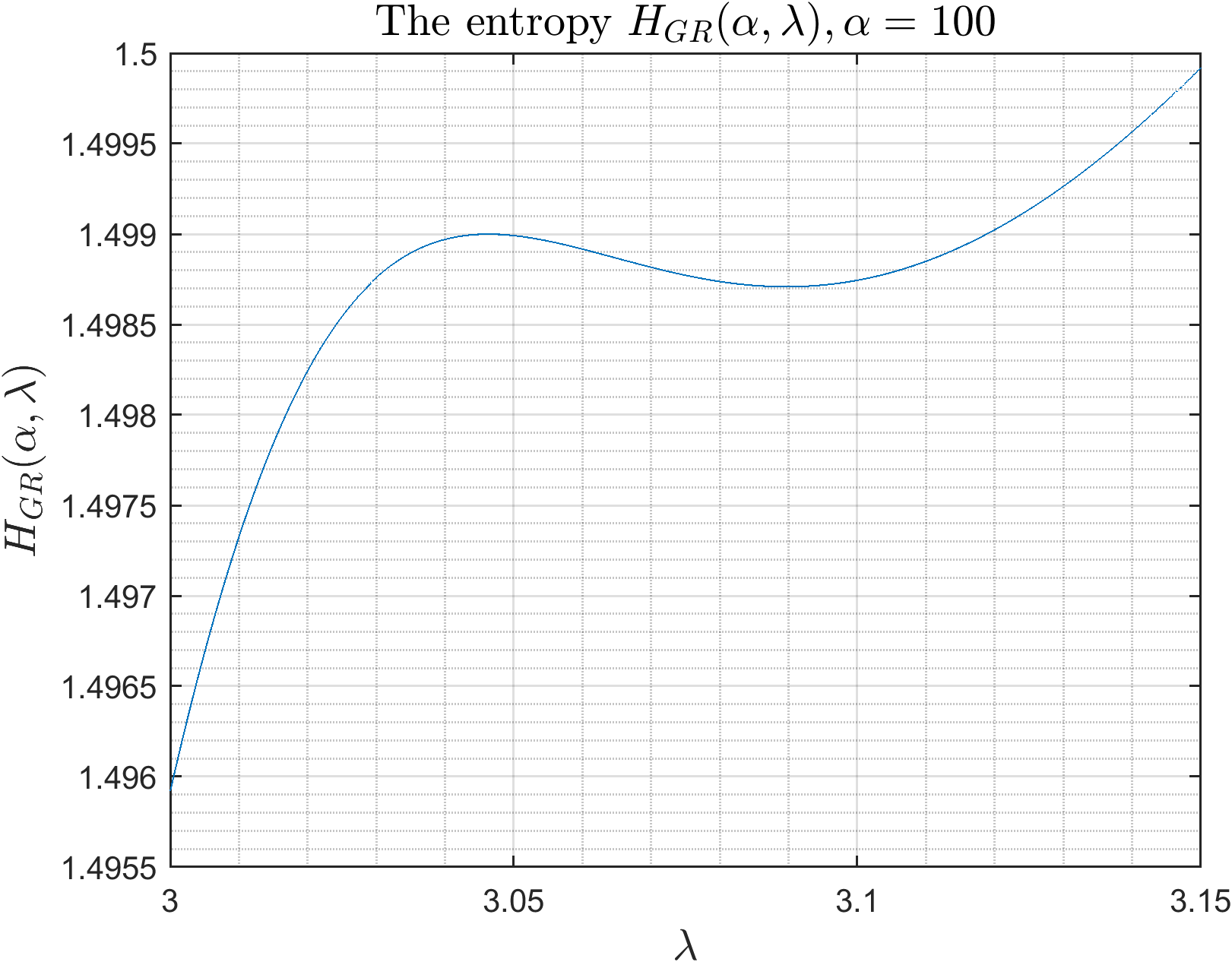}
    \caption{For $\alpha=100$, the monotonicity fails near $\lambda=1$, $\alpha=2$, and $\lambda=3$}
    \label{subfig:HGR1-alpha=100-zoom}
    \end{subfigure}
    \caption{The non-monotone behavior of $\entropy_{GR}(\alpha,\lambda)$ in $\lambda$ for large $\alpha$}
    \label{fig:HGR1-largealphas-all}
    \end{figure}   

\begin{figure}[!ht]
\centering
\includegraphics[width=0.8\textwidth]{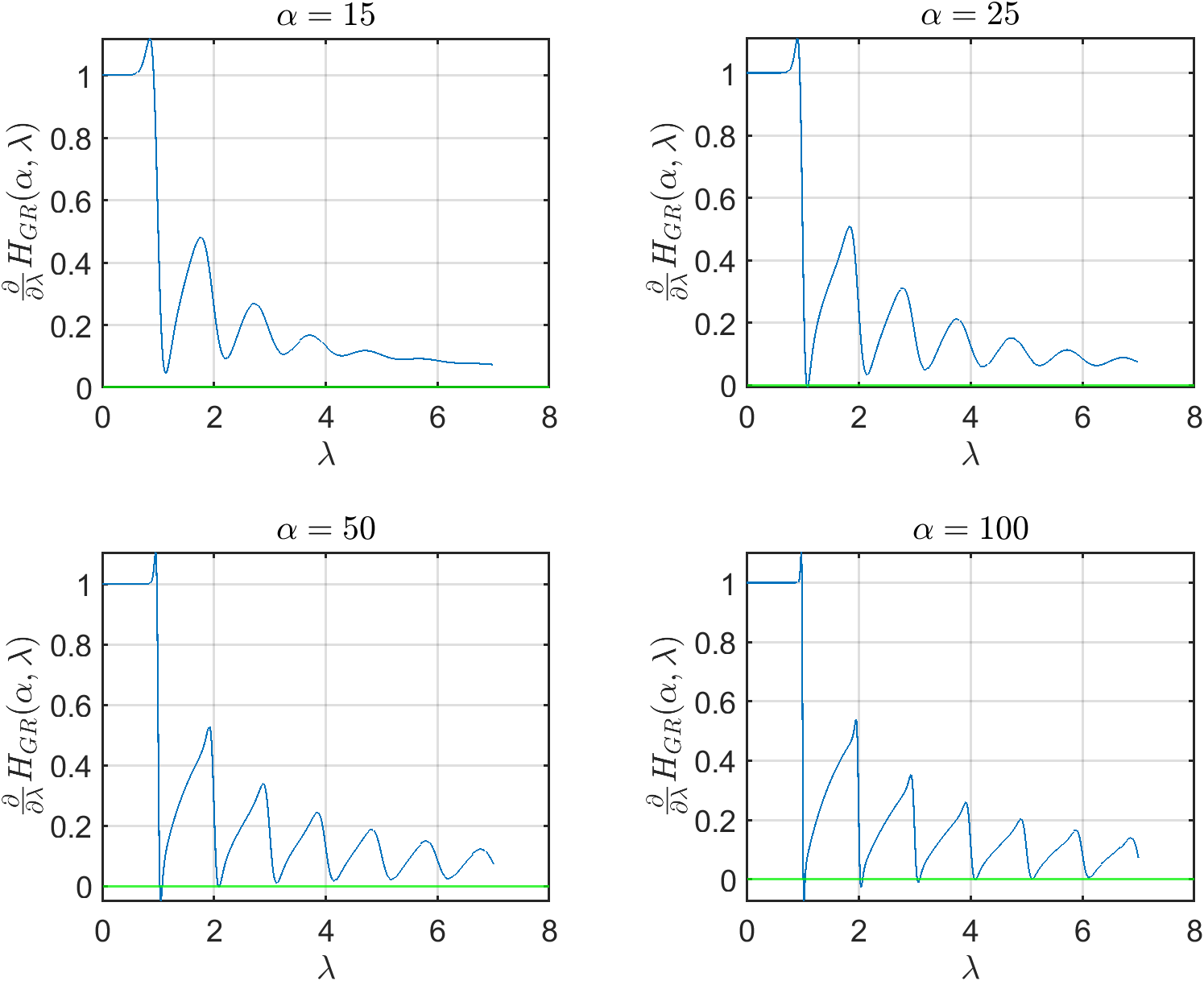}
\caption{Damping oscillations of $\frac{\partial}{\partial\lambda}\entropy_{GR}(\alpha, \lambda)$ in $\lambda$ for different values of $\alpha$}
\label{fig:derHGR1}
\end{figure}

\subsubsection{The generalized R\'enyi entropy $\entropy_{GR}(\alpha,\beta,\lambda)$}
Now we consider the generalized R\'enyi entropy with two parameters $\alpha$ and $\beta$.

\begin{proposition} \label{th:GR}
\begin{enumerate}
    \item Let $\alpha \le 1 < \beta$ or $\alpha > 1 \ge \beta$. Then $\entropy_{GR}(\alpha, \beta, \lambda)$ increases as a function of~$\lambda\in(0,\infty)$.
    \item Let $J$ be a subset of $(0,1)$ (as in Proposition~\ref{prop:anomalGR1zero}) or a subset of $(1,\infty)$ (see Remark~\ref{rem:anomGR1largealpha}) such that, for each $\alpha\in J$, the entropy $\entropy_{GR}(\alpha,\lambda)$ is decreasing in $\lambda$ on some interval (dependent on $\alpha$). Then, for each $\alpha\in J$ and for each $\beta$ in a neighborhood of $\alpha$, the entropy $\entropy_{GR}(\alpha,\beta,\lambda)$ is also decreasing in $\lambda$ on some interval (dependent on $\alpha,\beta$).    
\end{enumerate}
\end{proposition}

\begin{proof}
First, we note that, by \eqref{eq:entGGR-psi}, for any $\alpha,\beta>0$, $\alpha\neq\beta$,
\[
\entropy_{GR}(\alpha, \beta, \lambda)=\entropy_{GR}( \beta, \alpha, \lambda).
\]
Hence, without loss of generality, we assume that $\beta>\alpha$.
We rewrite then \eqref{eq:entGGR-psi} as follows
\begin{equation}\label{eq:HGRRrewr}
    \entropy_{GR}(\alpha, \beta, \lambda) =\frac{1}{\beta-\alpha}\bigl(\log\psi(\alpha,\lambda)-\log \psi(\beta,\lambda)\bigr),
\end{equation}
where, recall, $\psi$ is defined through \eqref{psifunction}.

1. If $\beta > 1 >\alpha$, then $\beta-\alpha>0$, and, by Proposition~\ref{prop:psi-monot}, $\log\psi(\alpha,\lambda)$ is increasing in $\lambda$, and $\log \psi(\beta,\lambda)$ is decreasing in $\lambda$. Therefore, by \eqref{eq:HGRRrewr}, $\entropy_{GR}(\alpha,\beta,\lambda)$ is increasing in $\lambda$. 
Similarly, if $\beta>1=\alpha$, then $\log\psi(1,\lambda)\equiv0$, and the same arguments work.

2. Let $\alpha < \beta < 1$ or $1<\alpha<\beta$.
By \eqref{eq:HGRRrewr}, the condition
$\frac\partial{\partial \lambda}\entropy_{GR}(\alpha,\beta,\lambda)\rvert_{\lambda=\lambda_0}<0$ for some $\lambda_0>0$ is equivalent to 
    \begin{equation}\label{dlambdamonot}
        \frac\partial{\partial \lambda}\log\psi(\alpha,\lambda)\Bigr\rvert_{\lambda=\lambda_0} < \frac\partial{\partial \lambda}\log\psi(\beta,\lambda)\Bigr\rvert_{\lambda=\lambda_0}.
    \end{equation}
A sufficient condition to have \eqref{dlambdamonot} for all $\alpha,\beta\in I$ with $\alpha<\beta$, for some $I\subset J$, would be to have that the function 
\[
I\ni \alpha\mapsto \frac\partial{\partial \lambda}\log\psi(\alpha,\lambda)\Bigr\rvert_{\lambda=\lambda_0} 
\]
is strictly increasing on $I$, that is equivalent to
\begin{align}
    0&<\frac{\partial}{\partial\alpha} \biggl(\frac\partial{\partial \lambda}\log\psi(\alpha,\lambda)\Bigr\rvert_{\lambda=\lambda_0} \biggr)=\Bigl(\frac{\partial^2}{\partial\alpha \partial \lambda}\log\psi(\alpha,\lambda)\Bigr)\biggr\rvert_{\lambda=\lambda_0}\notag\\
    &= \frac{\partial}{\partial\lambda} \Bigl(\frac\partial{\partial \alpha}\log\psi(\alpha,\lambda) \Bigr)\biggr\rvert_{\lambda=\lambda_0}
    =-\frac{\partial}{\partial\lambda}\entropy_{GR}(\alpha,\lambda)\Bigr\rvert_{\lambda=\lambda_0},\label{eq:3214234}
\end{align}
where the equalities between partial derivatives of $\log\psi$ hold by Lemma~\ref{l:der-psi} and Remark \ref{rem:mixed-der}, and we also used \eqref{eq:entGR-psi}. 

As a result, if, for some $\alpha_0\in J$, $\entropy_{GR}(\alpha_0,\lambda)$ is decreasing in $\lambda_0$ from a neighborhood of some $\lambda_0>0$, then the inequality $\frac{\partial}{\partial\lambda}\entropy_{GR}(\alpha,\lambda)\Bigr\rvert_{\lambda=\lambda_0}<0$ holds for all $\alpha$ from a neighborhood $I$ of $\alpha_0$. Then, by \eqref{eq:3214234}, the inequality \eqref{dlambdamonot} holds for all $\alpha,\beta\in I$, and thus for them, $\frac\partial{\partial \lambda}\entropy_{GR}(\alpha,\beta,\lambda)\rvert_{\lambda=\lambda_0}<0$. Since the functions in \eqref{eq:HGRRrewr} are continuously differentiable, we will get the statement.
\end{proof}

\begin{remark}
    Figure \ref{subfig:GR2small} presents the surface plot of $\entropy_{GR} (\alpha, \beta, \lambda)$ as a function of $\alpha$ and $\lambda$, with fixed $\beta = 0.01$, within the region $0.01 \le \alpha \le 0.03$ and $0 < \lambda \le 150$. In Figure \ref{subfig:GR2smallboth}, the dependence of $\entropy_{GR} (\alpha, \beta, \lambda)$ on $\lambda$ is shown for fixed values of $\alpha = 0.02$ and $\beta = 0.01$. The results indicate an ``anomalous'' behavior of the generalized Rényi entropy for these values of $\alpha$ and $\beta$: initially, the entropy decreases with increasing $\lambda$, followed by a subsequent increase. This supports the second statement of Proposition~\ref{th:GR} in view of the similar results observed in Figure~\ref{fig:GRplots}. 

    The proof of the second statement of Proposition~\ref{th:GR} implies that the behavior of $\entropy_{GR} (\alpha, \beta, \lambda)$ for large $\alpha$ and $\beta$ should be similar to the one observed in Figure~\ref{fig:HGR1-largealphas-all}. Indeed Figure~\ref{fig:HGR2} supports this statement.
    
\end{remark}

\begin{figure}[bht]
    \centering
    \begin{subfigure}{0.48\textwidth}
        \includegraphics[width=\textwidth]{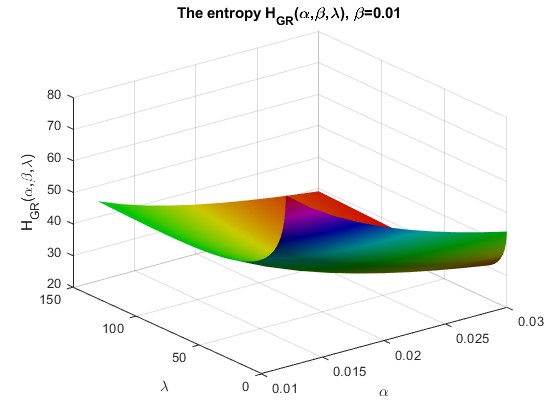}
        \caption{}
        \label{subfig:GR2small}
    \end{subfigure}
    \hfill
    \begin{subfigure}{0.48\textwidth}
        \includegraphics[width=\textwidth]{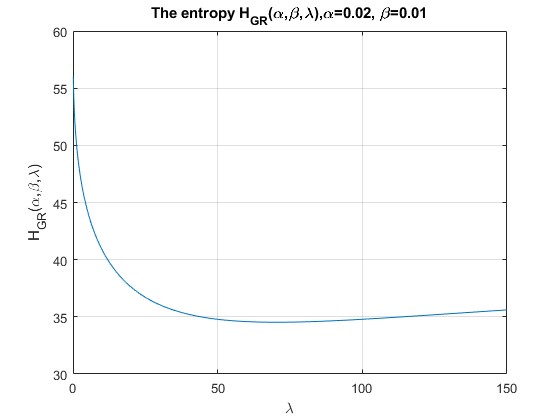}
        \caption{}
        \label{subfig:GR2smallboth}
    \end{subfigure}
    \caption{The generalized R\'enyi entropy $\entropy_{GR}(\alpha,\beta,\lambda)$: (a)~as a function of $(\alpha,\lambda)$ for $\beta = 0.01$, (b)~as a function of $\lambda$ for $\alpha = 0.02$ and $\beta = 0.01$}
    \label{fig:GR2}
\end{figure}

\begin{figure}[bht]
 \centering
    \begin{subfigure}{0.48\textwidth}
        \includegraphics[width=\textwidth]{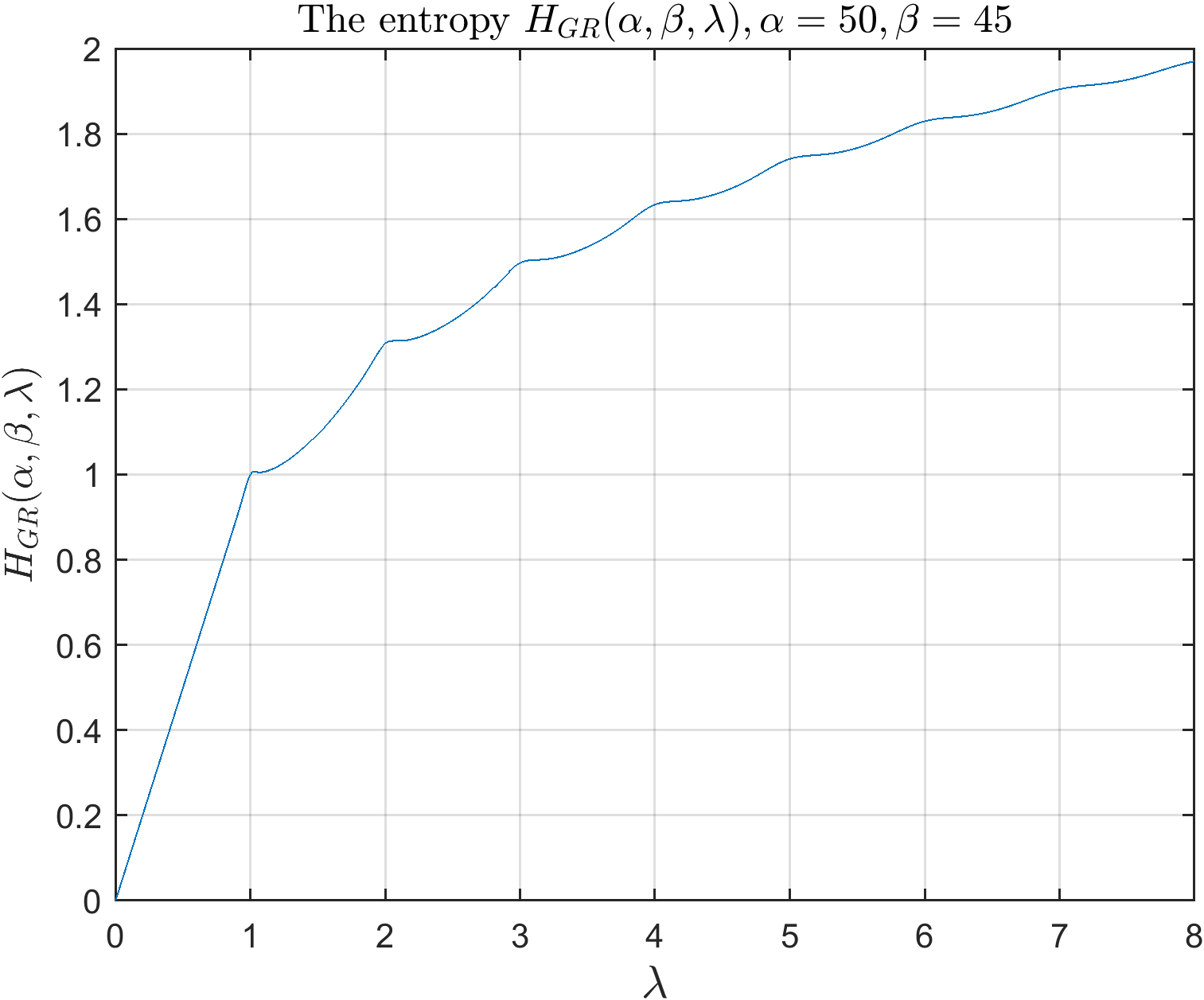}
        \caption{}
        \label{subfig:GR2large}
    \end{subfigure}
    \hfill
    \begin{subfigure}{0.48\textwidth}
        \includegraphics[width=\textwidth]{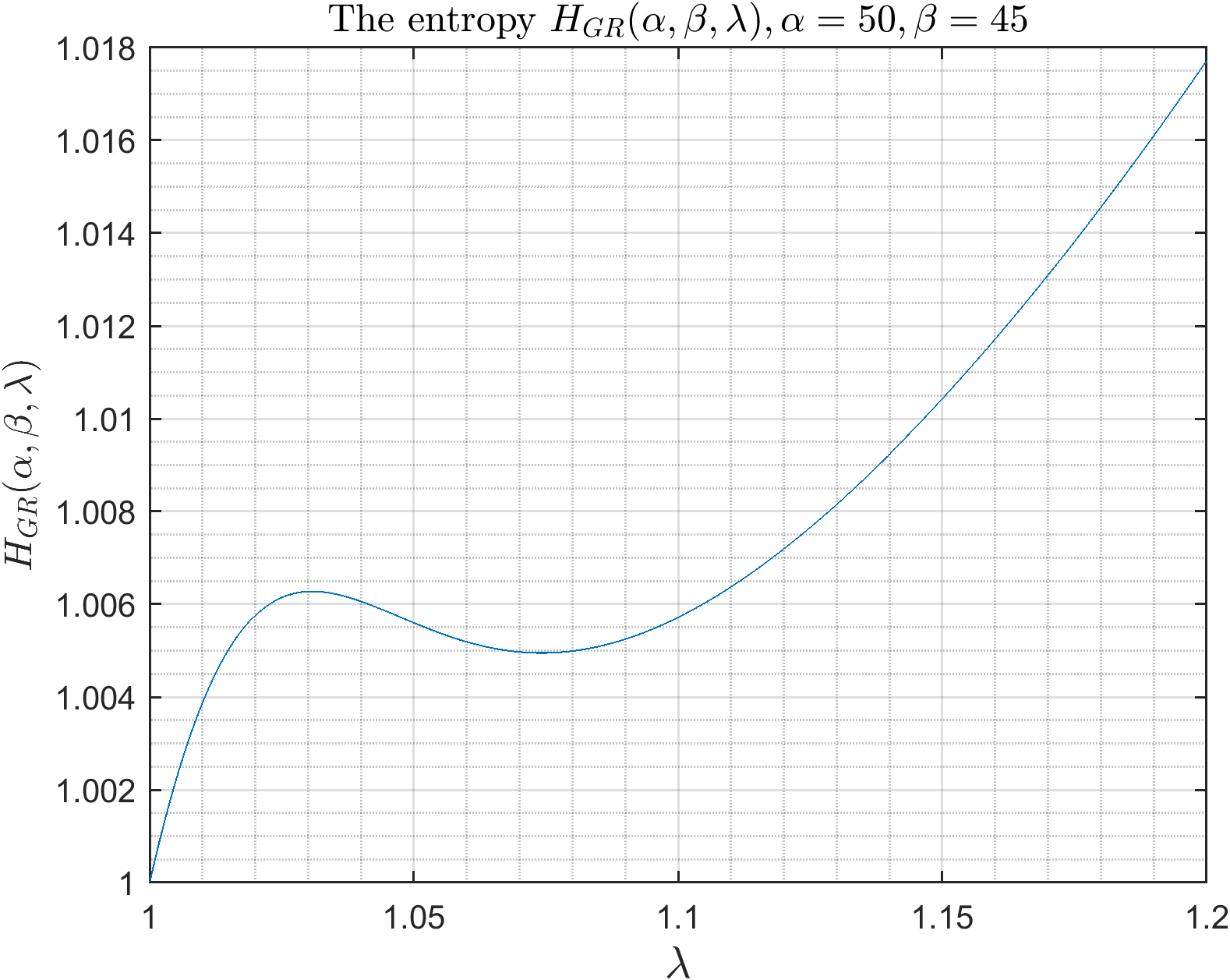}
        \caption{}
        \label{subfig:GR2largezoom}
    \end{subfigure}
    \caption{The generalized R\'enyi entropy $\entropy_{GR}(\alpha,\beta,\lambda)$ for large $\alpha$ and $\beta$}
    \label{fig:HGR2}
\end{figure}

\appendix

\section{Appendix}  

\subsection{Auxiliary results}

\begin{lemma}\label{l:series-to-inf}
For any $\lambda>0$,
\[
\lim_{\lambda\to\infty} e^{-\lambda} \sum_{i=1}^\infty \frac{\lambda^i}{i!} \log(i+1) = \infty.
\]
\end{lemma}

\begin{proof}
Obviously, for any $N>1$
\[
\sum_{i=1}^\infty \frac{\lambda^i}{i!} \log(i+1)
\ge
\log(N+1) \sum_{i=N}^\infty \frac{\lambda^i}{i!} .
\]
Furthermore,
\[
\lim_{\lambda\to\infty} e^{-\lambda} \sum_{i=0}^{N-1} \frac{\lambda^i}{i!}   = 0.
\]
Therefore,
\[
\lim_{\lambda\to\infty} e^{-\lambda} \sum_{i=1}^\infty \frac{\lambda^i}{i!} \log(i+1) 
\ge
 \log(N+1) \lim_{\lambda\to\infty} e^{-\lambda} \sum_{i=N}^\infty \frac{\lambda^i}{i!} =  \log(N+1) \lim_{\lambda\to\infty} e^{-\lambda} \sum_{i=0}^\infty \frac{\lambda^i}{i!}= \log(N+1).
\]
Since $N>1$ is arbitrary, we get the statement.
\end{proof}

\begin{lemma}\label{l:max-poisson}
For any $\lambda>0$,
\begin{equation}\label{eq:mulambda}
    \mu(\lambda):=\max_{i\ge 0} p_i(\lambda) 
    = \frac{\lambda^{\lfloor\lambda\rfloor}}{\lfloor\lambda\rfloor!}   e^{-\lambda}, 
\end{equation}
and, for $\lambda\geq1$,
\begin{equation}\label{eq:mulambdale1overla}
    \mu(\lambda)<\frac{1}{\sqrt{2\pi \lfloor\lambda\rfloor}}.
\end{equation}
In particular, the maximal probability of the Poisson distribution with parameter $\lambda > 0$ tends to zero as $\lambda \rightarrow \infty$. Moreover, the following modification of the estimate \eqref{eq:mulambdale1overla} holds true: 
\begin{equation}\label{eq:mulambdale1overlaenhanced}
    \mu(\lambda)<\frac{1}{\sqrt{2\pi \lambda}}\left(1+\frac{1}{(\lambda-1)\vee 1}\right)^{\frac12}\exp\left(-\frac1{12\lambda+1}\right), \qquad\lambda>1.
\end{equation}

\end{lemma}

\begin{proof}

To prove \eqref{eq:mulambda}, consider the indices $i$ for which the sequence $p_i(\lambda)$ is increasing. Solving the inequality $p_{i-1}(\lambda) \leq p_i(\lambda)$, that is,
\[
\frac{\lambda^{i-1}   e^{-\lambda}}{(i-1)!} \leq \frac{\lambda^i   e^{-\lambda}}{i!},
\]
we get $i \leq \lambda$.
Hence,
\begin{equation*}
\begin{cases}
p_{i-1}(\lambda) \leq p_i(\lambda) &\text{for } i \leq \lambda,\\
p_{i+1}(\lambda) \leq p_i(\lambda) &\text{for } i \geq \lambda - 1.
\end{cases}
\end{equation*}
This means that the maximum value of $p_i$ is reached at $i \in [\lambda - 1, \lambda]$.
    
Let us consider two cases: $\lambda \notin \mathbb{N}$ and $\lambda \in \mathbb{N}$.
    
\emph{Case 1:} $\lambda \notin \mathbb{N}$. Since $i \in \mathbb{N} \cup \{0\}$, the only $i \in [\lambda - 1, \lambda]$ is $i = \lfloor\lambda\rfloor$. That is, the maximum $p_i(\lambda)$ in the case $\lambda \notin \mathbb{N}$ is reached at $i = \lfloor\lambda\rfloor$, i.e.,
\begin{equation}\label{eq:entite}
\max_{i\ge 0} p_i(\lambda) = p_{\lfloor\lambda\rfloor} (\lambda) = \frac{\lambda^{\lfloor\lambda\rfloor}}{\lfloor\lambda\rfloor!}   e^{-\lambda}.
\end{equation}
    
\emph{Case 2:} $\lambda \in \mathbb{N}$. Since $i \in \mathbb{N} \cup \{0\}$, the possible values of $i \in [\lambda - 1, \lambda]$ are $i = \lambda - 1$ and $i = \lambda$. That is, the maximum value of $p_i(\lambda)$ in the case $\lambda \in \mathbb{N}$ is reached at $i = \lambda - 1$ or $i = \lambda$. Let us compute the values of $p_i(\lambda)$ for $i = \lambda - 1$ and $i = \lambda$:
\begin{gather*}
p_{\lambda -1} (\lambda) = \frac{\lambda^{\lambda-1}}{(\lambda-1)!}   e^{-\lambda}
= \frac{\lambda^\lambda}{\lambda!}  e^{-\lambda}
= p_{\lambda}(\lambda).
\end{gather*}
Thus, for $\lambda \in \mathbb{N}$ the relation \eqref{eq:entite} also holds.

To estimate $\mu(\lambda)$, we start with the well-know inequality
\begin{equation}\label{eq:stirlingest}
    \sqrt{2\pi n} \Bigl( \frac{n}{e}\Bigr)^ne^{\frac{1}{12n+1}}<n!<\sqrt{2\pi n} \Bigl( \frac{n}{e}\Bigr)^ne^{\frac{1}{12n}}, \quad n\geq1.
\end{equation}
Then, by the first inequality in \eqref{eq:stirlingest},
\begin{equation}
    \frac{\lambda^{\lfloor\lambda\rfloor}}{\lfloor\lambda\rfloor!}   e^{-\lambda}<
    \frac{1}{\sqrt{2\pi \lfloor\lambda\rfloor}} \Bigl( \frac{\lambda}{\lfloor\lambda\rfloor}\Bigr)^{\lfloor\lambda\rfloor}  e^{\lfloor\lambda\rfloor}e^{-\lambda}\exp\left(-\frac{1}{12\lfloor\lambda\rfloor+1}\right), \qquad \lambda\geq1.\label{eq:ax1}
\end{equation}

We consider, for a fixed $\lambda \geq 1$, the function
$f(x) = (\frac{\lambda}{x})^{x} e^x$, $x \in (\lambda - 1, \lambda]$.
Its logarithm equals
\[
g(x) = \log f(x) = x \log\lambda - x \log x + x,
\]
and
$g'(x) =  \log\lambda - \log x \ge 0$.
Therefore, the highest value of $f(x)$ is achieved at $x = \lambda$ and equals $f(\lambda) = e^\lambda$. As a result, 
\begin{equation*}
    \left(\frac{\lambda}{\lfloor\lambda\rfloor}\right)^{\lfloor\lambda\rfloor}
e^{\lfloor\lambda\rfloor }e^{- \lambda} \le 1, \qquad \lambda\geq1,
\end{equation*}
and hence, \eqref{eq:ax1} reads
\begin{equation}
    \frac{\lambda^{\lfloor\lambda\rfloor}}{\lfloor\lambda\rfloor!}   e^{-\lambda}<
    \frac{1}{\sqrt{2\pi \lfloor\lambda\rfloor}} \exp\left(-\frac{1}{12\lfloor\lambda\rfloor+1}\right), \qquad \lambda\geq1,\label{eq:ax2}
\end{equation}
that implies \eqref{eq:mulambdale1overla}. 

To get \eqref{eq:mulambdale1overlaenhanced}, we note that, in \eqref{eq:ax2},
\[
\exp\left(-\frac{1}{12\lfloor\lambda\rfloor+1}\right) < \exp\left(-\frac{1}{12\lambda+1}\right);
\]
and
\begin{align*}
    \frac{1}{\sqrt{2\pi \lfloor\lambda\rfloor}} &=\frac{1}{\sqrt{2\pi \lambda}} \left(\frac{\lambda}{\lfloor\lambda\rfloor}\right)^{\frac12}
    <\frac{1}{\sqrt{2\pi \lambda}} \left(\frac{\lfloor\lambda\rfloor+1}{\lfloor\lambda\rfloor}\right)^{\frac12}
    =\frac{1}{\sqrt{2\pi \lambda}} \left(1+\frac{1}{\lfloor\lambda\rfloor}\right)^{\frac12}\\
    &<\frac{1}{\sqrt{2\pi \lambda}} \left(1+\frac{1}{(\lambda-1)\vee 1}\right)^{\frac12}<\frac{1}{\sqrt{2\pi \lambda}} \exp\left(\frac{1}{2((\lambda-1)\vee 1)}\right),
\end{align*}
where we used the inequality $1+x < e^x$ for $x>0$.
\end{proof}

\begin{lemma}\label{l:lim-ratio}
\begin{equation}\label{eq:lim-ratio}
\lim_{\alpha\downarrow0} \frac{\alpha \sum_{i=1}^\infty \frac{i}{(i!)^\alpha}}{\sum_{i=0}^\infty \frac{1}{(i!)^\alpha}} = 0.
\end{equation}
\end{lemma}

\begin{proof}
Fix $\varepsilon > 0$. We decompose the sum in the numerator into two terms as follows
\[
\sum_{i=1}^\infty \frac{i}{(i!)^\alpha}
= \sum_{1\le i\le \varepsilon/\alpha}\frac{i}{(i!)^\alpha}
+ \sum_{i > \varepsilon/\alpha}\frac{i}{(i!)^\alpha},
\quad \alpha<\varepsilon.
\]
The first term admits the bound:
\[
\alpha \sum_{1\le i\le \varepsilon/\alpha}\frac{i}{(i!)^\alpha}
\le \varepsilon \sum_{1\le i\le \varepsilon/\alpha}\frac{1}{(i!)^\alpha}
\le \varepsilon \sum_{i=0}^\infty \frac{1}{(i!)^\alpha},
\]
whence
\begin{equation}\label{eq:ratio-1term}
\frac{\alpha \sum_{1\le i\le \varepsilon/\alpha}\frac{i}{(i!)^\alpha}}{\sum_{i=0}^\infty \frac{1}{(i!)^\alpha}} \le \varepsilon.
\end{equation}
Now, let us consider the second term.
By the first inequality in \eqref{eq:stirlingest}, $i! > i^{i+\frac12} e^{-i}$, and therefore,
\begin{align*}
\alpha\sum_{i > \varepsilon/\alpha}\frac{i}{(i!)^\alpha}
&\le \alpha \sum_{i = \lfloor\varepsilon/\alpha\rfloor + 1}^\infty i^{1 - \alpha i - \frac{\alpha}{2}} e^{i\alpha}
\\
&= \alpha\sum_{k = 1}^\infty \left(k + \left\lfloor\frac{\varepsilon}{\alpha}\right\rfloor\right)^{1 - \alpha k - \alpha \lfloor\varepsilon/\alpha\rfloor - \frac{\alpha}{2}} \exp\set{\alpha k + \alpha \left\lfloor\frac{\varepsilon}{\alpha}\right\rfloor}
\\
&\le \alpha\sum_{k = 1}^\infty \frac{k + \left\lfloor\frac{\varepsilon}{\alpha}\right\rfloor}{\left(k + \left\lfloor\frac{\varepsilon}{\alpha}\right\rfloor\right)^{\alpha k + \alpha \lfloor\varepsilon/\alpha\rfloor + \frac{\alpha}{2}}} e^{\alpha k + \varepsilon}
\le \alpha\sum_{k = 1}^\infty \frac{k + \frac{\varepsilon}{\alpha}}{\left(\frac{\varepsilon}{\alpha}\right)^{\alpha k + \alpha \lfloor\varepsilon/\alpha\rfloor + \frac{\alpha}{2}}} e^{\alpha k + \varepsilon}
\\
&= \left(\frac{\alpha}{\varepsilon}\right)^{\alpha \left\lfloor\frac{\varepsilon}{\alpha}\right\rfloor + \frac{\alpha}{2}} 
e^\varepsilon\sum_{k = 1}^\infty \left(\alpha k + \varepsilon\right)\left(\frac{\alpha e}{\varepsilon}\right)^{\alpha k }.
\end{align*}
For $\alpha < \frac{\varepsilon}{e}$ the series can be computed using the following relations:
\[
\sum_{k = 1}^\infty x^k = \frac{x}{1-x},\quad
\sum_{k = 1}^\infty k x^k = x\left(\sum_{k = 0}^\infty x^k\right)'
= x\left(\frac{1}{1-x}\right)' = \frac{x}{(1-x)^2}, \quad \abs{x}<1.
\]
We obtain
\begin{equation}\label{eq:numer-upbound}
\alpha\sum_{i > \varepsilon/\alpha}\frac{i}{(i!)^\alpha} 
\le \left(\frac{\alpha}{\varepsilon}\right)^{\alpha \left\lfloor\frac{\varepsilon}{\alpha}\right\rfloor + \frac{\alpha}{2}} 
e^\varepsilon  \left(\frac{\alpha e}{\varepsilon}\right)^{\alpha}
\left( \frac{\alpha}{\left(1 - \left(\frac{\alpha e}{\varepsilon}\right)^{\alpha}\right)^2} +  \frac{\varepsilon}{1 - \left(\frac{\alpha e}{\varepsilon}\right)^{\alpha}}\right).
\end{equation}
Observe that
\[
1 - \left(\frac{\alpha e}{\varepsilon}\right)^{\alpha}
= 1 - \exp\set{\alpha\log\frac{\alpha e}{\varepsilon}}
\sim - \alpha\log\frac{\alpha e}{\varepsilon}
\sim - \alpha\log \alpha, \quad \text{as } \alpha\downarrow0.
\]
Therefore, the right-hand side of \eqref{eq:numer-upbound} is asymptotically equivalent to
\[
\left(\frac{\alpha e}{\varepsilon}\right)^{\varepsilon}
\left( \frac{1}{\alpha\log^2\alpha} -  \frac{\varepsilon}{\alpha\log\alpha}\right)
\sim - \frac{e^\varepsilon \varepsilon^{1-\varepsilon}}{\alpha^{1-\varepsilon} \log\alpha}, \quad \text{as } \alpha\downarrow0.
\]
Now let us study the denominator of \eqref{eq:lim-ratio}. By the second inequality in \eqref{eq:stirlingest}, we may write for any $N\ge1$,
\begin{align*}
\sum_{i=0}^\infty \frac{1}{(i!)^\alpha}
&\ge \sum_{i=1}^N (2\pi i )^{-\frac{\alpha}{2}} i^{-\alpha i} e^{\alpha \left(i-\frac{1}{12i}\right)}
\ge  (2\pi )^{-\frac{\alpha}{2}} \sum_{i=1}^N  i^{-\frac{\alpha}{2}-\alpha i} 
\ge  (2\pi )^{-\frac{\alpha}{2}} \sum_{i=1}^N  N^{-\frac{\alpha}{2}-\alpha i} 
\end{align*}
Taking $N = \lfloor \frac1\alpha\rfloor < \frac1\alpha$, we get
\begin{align*}
\sum_{i=0}^\infty \frac{1}{(i!)^\alpha}
\ge  (2\pi )^{-\frac{\alpha}{2}} \sum_{i=1}^{\lfloor \frac1\alpha\rfloor}  \alpha^{\frac{\alpha}{2}+\alpha i} 
=  (2\pi )^{-\frac{\alpha}{2}} \alpha^{\frac{3\alpha}{2}}  \frac{1 - \alpha^{\alpha \lfloor \frac1\alpha\rfloor }}{1 - \alpha^{\alpha}}
\sim -\frac{1}{\alpha\log\alpha},
\quad\text{as } \alpha\downarrow 0.
\end{align*}
Combining the results above we see that the ratio $\alpha\sum_{i > \varepsilon/\alpha}\frac{i}{(i!)^\alpha} / \sum_{i=0}^\infty \frac{1}{(i!)^\alpha}$
is bounded from above by the function, asymptotically equivalent to
$e^\varepsilon \varepsilon^{1-\varepsilon}\alpha^{\varepsilon} \to0$, as $\alpha\downarrow0$. 
Together with \eqref{eq:ratio-1term} this implies that
\[
\lim_{\alpha\downarrow0} \frac{\alpha \sum_{i=1}^\infty \frac{i}{(i!)^\alpha}}{\sum_{i=0}^\infty \frac{1}{(i!)^\alpha}} \le\varepsilon,
\]
for any $\varepsilon>0$.
Hence \eqref{eq:lim-ratio} holds.
\end{proof}

\subsection{Some properties of the function $\psi(\alpha, \lambda)$}

\begin{lemma}\label{l:der-psi}
    For all $\alpha>0$ and $\lambda>0$, the function $\psi(\alpha, \lambda)$ introduced by \eqref{psifunction} is well defined and continuously differentiable in two variables, moreover, the partial derivatives can be computed as follows:
    \begin{gather}
    \label{eq:der-psi-alpha}
    \frac{\partial}{\partial \alpha} \psi(\alpha, \lambda) = 
    \sum_{i=0}^\infty (p_i(\lambda))^\alpha \log p_i(\lambda),
    \\
    \label{eq:der-psi-lambda}
    \frac{\partial}{\partial \lambda} \psi(\alpha, \lambda) = \alpha e^{-\alpha \lambda} \sum_{i=0}^\infty (i-\lambda) \frac{\lambda^{\alpha i - 1}}{(i!)^\alpha},
    \end{gather}
    where $p_i(\lambda) = e^{-\lambda} \frac{\lambda^i}{i!}$.
\end{lemma}

\begin{proof}
It is well known (see, for example, \cite[Section 16.3.5]{Zorich-part2}) that to verify the term-by-term differentiability of a series with differentiable terms, it suffices to establish that the series converges at least at one point and that the corresponding series of derivatives converges uniformly.

Note that the convergence of the series \eqref{psifunction} (as well as the series \eqref{eq:der-psi-lambda}) for all positive values of $\alpha$ and $\lambda$ follows from, for instance, the root test, using the bound \eqref{eq:stirlingest}. 

Next, we prove the uniform convergence of the series of derivatives, i.e., the series \eqref{eq:der-psi-alpha} and \eqref{eq:der-psi-lambda}.
Consider an arbitrary rectangle $R=[\alpha_*,\alpha^*]\times[\lambda_*,\lambda^*]\subset(0,\infty)^2$.
For any $\delta > 0$, there exists a constant $C_\delta > 0$ such that $\abs{\log x} \leq C_\delta x^{-\delta}$ for all $x \in (0, 1)$. Therefore, choosing $\delta \in(0, \alpha_*)$, we can estimate
\[
(p_i(\lambda))^\alpha \abs{\log p_i(\lambda)} \le C_\delta (p_i(\lambda))^{\alpha_* - \delta}
\le \left(e^{-\lambda_*} \frac{(\lambda^*)^i}{i!}\right)^{\alpha_* - \delta}
= e^{(\alpha_* - \delta) (\lambda^*-\lambda_*)} (p_i(\lambda^*))^{\alpha_* - \delta},
\]
for all $(\alpha,\lambda) \in R$.
Thus, by the Weierstrass M-test, the series \eqref{eq:der-psi-alpha} converges uniformly on $R$ because the series $\sum_{i\ge0} (p_i(\lambda^*))^{\alpha_* - \delta}$ converges to $\psi(\alpha_* - \delta,\lambda^*)$.

The uniform convergence of the series \eqref{eq:der-psi-lambda} on $R$ can be established similarly. Specifically, the $i$th  term can be bounded as follows:
\[
\abs{\alpha e^{-\alpha \lambda} (i-\lambda) \frac{\lambda^{\alpha i - 1}}{(i!)^\alpha}}
\le \alpha\left(\frac{i}{\lambda} +1\right) (p_i(\lambda))^\alpha
\le \alpha^*\left(\frac{i}{\lambda_*} +1\right) e^{(\alpha_* - \delta) (\lambda^*-\lambda_*)} (p_i(\lambda^*))^{\alpha_* - \delta}
\eqqcolon M_i,
\]
and the convergence of the series $\sum_i M_i$ follows from the root test similarly to the series \eqref{psifunction}.

The uniform convergence results imply that the series  $\psi(\alpha, \lambda)$ can be differentiated term-by-term with respect to both $\alpha$ and $\lambda$, and the formulas \eqref{eq:der-psi-alpha} and \eqref{eq:der-psi-lambda} hold for any $(\alpha, \lambda) \in R$. Furthermore, $\frac{\partial}{\partial \alpha} \psi(\alpha, \lambda)$ and $\frac{\partial}{\partial \lambda} \psi(\alpha, \lambda)$ are continuous on $R$. Since $\alpha^* > \alpha_* > 0$ and $\lambda^* > \lambda_* > 0$ are chosen arbitrarily, we conclude that the result holds for all $\alpha > 0$ and $\lambda > 0$.
\end{proof}
\begin{remark}\label{rem:mixed-der}
1. Arguing as above, one can show that the mixed derivative $\frac{\partial^2}{\partial\lambda\partial\alpha} \psi(\alpha,\lambda)$ is also continuous in two variables and it can be found using term-by-term differentiation of the series $\psi(\alpha,\lambda)$.

2. The term-by-term differentiability of $\psi(\alpha, \lambda)$ with respect to $\lambda$ can be also deduced from the fact that it is a power series in terms of $\lambda^\alpha$, and any power series can be differentiated term-by-term within its region of convergence.
\end{remark}

\begin{lemma}\label{le:estforpsi}
Let $\gamma_*:=\exp\left(-\dfrac{\pi}{e}\bigl(e^{\frac16}-1\bigr)\right)\approx 0.811$. 
For each $\alpha>0$, $\alpha\neq1$ and each $\gamma\in(\gamma_*,1)$, we define
\begin{equation*}\label{eq:defD}
    D(\alpha,\gamma):=\left(\frac{\pi|\alpha-1|}{-e\alpha\log\gamma}\right) ^{\frac{|\alpha-1|}{2}}.
\end{equation*}
Then
\begin{enumerate}[label=\arabic*)]
    \item for $\alpha\in(0,1)$,
    \begin{equation*}
        \label{eq:psiupperestless1}
        \psi(\alpha,\lambda) < C_1(\alpha,\gamma)e^{-\alpha\lambda}E_\alpha\left(\Bigl(\dfrac{\alpha}{\gamma}\lambda\Bigr)^{\alpha } \right), \quad\lambda>0,
    \end{equation*}
where
    \begin{equation*}
        \label{eq:defC1}
        C_1(\alpha,\gamma):=\sqrt{\alpha}e^{\frac{1}{12\alpha}} D(\alpha,\gamma);
    \end{equation*}
\item for $\alpha>1$, 
\begin{equation*}    \label{eq:psilowerestbig1}
    \psi(\alpha,\lambda)> 
C_2(\alpha,\gamma)e^{-\alpha\lambda}E_\alpha\left(\bigl(\alpha\gamma\lambda\bigr)^{\alpha } \right),    
\end{equation*}
where
    \begin{equation*}
        \label{eq:defC2}
        C_2(\alpha,\gamma) := \sqrt{\alpha} e^{-\frac{\alpha }{12}}\frac{1}{D(\alpha,\gamma)}.
    \end{equation*}
\end{enumerate}

\end{lemma}

\begin{proof}
It is well known that, for $x>0$,
\[
\sqrt{2\pi x}\left(\frac{x}{e}\right)^xe^{\frac{1}{12x+1}}<\Gamma(x+1)<\sqrt{2\pi x} \left(\frac{x}{e}\right)^xe^{\frac{1}{12x}}.
\]
Therefore, for  $\alpha>0$
\[
\sqrt{2\pi \alpha x}\left(\frac{\alpha x}{e}\right)^{\alpha x}e^{\frac{1 }{12\alpha x+1}}<\Gamma(\alpha x+1)<\sqrt{2\pi \alpha x} \left(\frac{\alpha x}{e}\right)^{\alpha x}e^{\frac{1}{12\alpha x}},
\]
and
\[
(\sqrt{2\pi x})^{-\alpha} \left(\frac{x}{e}\right)^{-\alpha x}e^{-\frac{\alpha }{12x}}<\frac1{\bigl(\Gamma(x+1)\bigr)^\alpha}<(\sqrt{2\pi x})^{-\alpha} \left(\frac{x}{e}\right)^{-\alpha x}e^{-\frac{\alpha }{12x+1}};
\]
hence,
\begin{equation}
    \sqrt{\alpha} (\sqrt{2\pi x})^{1-\alpha}  \bigl(\alpha^\alpha\bigr)^x   e^{\frac{1 }{12\alpha x+1}-\frac{\alpha }{12x}}< \frac{\Gamma(\alpha x+1)}{\bigl(\Gamma(x+1)\bigr)^\alpha}
<\sqrt{\alpha} (\sqrt{2\pi x})^{1-\alpha}\bigl(\alpha^\alpha\bigr)^x 
e^{\frac{1}{12\alpha x}-\frac{\alpha }{12x+1}}.\label{eq:ggest}
\end{equation}

For each $\alpha>0$, $\alpha\neq1$, and each $\gamma\in(0,1)$, we set $b:=\frac{|\alpha-1|}{2}>0$, and consider the function
$f_{\alpha,\gamma}(x):=x^{b}\gamma^{\alpha x}$, $x\geq0$, then $g_{\alpha,\gamma}(x):=\log f_{\alpha,\gamma}(x)=b\log x + x \alpha \log \gamma$, and $g_{\alpha,\gamma}'(x)=\frac{b}{x}+\alpha \log \gamma$, and hence, $f_{\alpha,\gamma}$ attains its maximum on $[0,\infty)$ at $\frac{b}{-\alpha\log \gamma}=\frac{|\alpha-1|}{-2\alpha\log \gamma}>0$ and hence, for $x>0$,
\[
x^{\frac{|\alpha-1|}{2}}\gamma^{\alpha x}\leq \left(\frac{|\alpha-1|}{-2e\alpha\log\gamma}\right) ^{\frac{|\alpha-1|}{2}}:=C_0(\alpha,\gamma).
\]

1. Let $0<\alpha<1$. For any $\gamma\in(\alpha,1)$, we have from the second inequality in \eqref{eq:ggest} that, for $x\geq1$,
\begin{align*}
    \frac{\Gamma(\alpha x+1)}{\bigl(\Gamma(x+1)\bigr)^\alpha}
&<\sqrt{\alpha} (\sqrt{2\pi})^{1-\alpha} x^{\frac{1-\alpha}2}\gamma^{\alpha x} \Bigl(\frac{\alpha}{\gamma}\Bigr)^{\alpha x} 
e^{\frac{1}{12\alpha}}\notag \\&\leq \sqrt{\alpha}e^{\frac{1}{12\alpha}} \bigl(\sqrt{2\pi}\bigr)^{1-\alpha}C_0(\alpha,\gamma)\Bigl(\frac{\alpha}{\gamma}\Bigr)^{\alpha x}=
C_1(\alpha,\gamma) \Bigl(\frac{\alpha}{\gamma}\Bigr)^{\alpha x},
\label{eq:c1est}
\end{align*}
since $\bigl(\sqrt{2\pi}\bigr)^{1-\alpha}C_0(\alpha,\gamma)=D(\alpha,\gamma)$ for $\alpha\in(0,1)$. 
Note that $C_0(\alpha,\gamma)\to\infty$ as $\gamma\nearrow1$, hence, we can assume that $\gamma$ is close enough to $1$ to ensure that $C_1(\alpha,\gamma)\geq1$ (for the given $\alpha)$. Then
\begin{align*}
    \psi(\alpha,\lambda) &= e^{-\alpha\lambda}+ e^{-\alpha\lambda}\sum_{i=1}^\infty\frac{\lambda^{\alpha i}}{(i!)^\alpha}
\leq e^{-\alpha\lambda}+e^{-\alpha\lambda}\sum_{i=1}^\infty\frac{\lambda^{\alpha i}}{\Gamma(\alpha i+1)}C_1(\alpha,\gamma)\Bigl(\frac{\alpha}{\gamma}\Bigr)^{\alpha i}
\\&\leq C_1(\alpha,\gamma)E_\alpha\left(\Bigl(\dfrac{\alpha}{\gamma}\lambda\Bigr)^{\alpha } \right),
\end{align*}
where $E_\alpha$ is the Mittag--Leffler function. 

2. Let $\alpha>1$. For any $\gamma\in(0,1)$, we have from the second inequality in \eqref{eq:ggest} that, for $x\geq1$,
\begin{align*}
   \frac{\Gamma(\alpha x+1)}{\bigl(\Gamma(x+1)\bigr)^\alpha}&>\sqrt{\alpha} (\sqrt{2\pi })^{1-\alpha} x^{\frac{1-\alpha}{2}} \alpha^{\alpha x}   e^{-\frac{\alpha }{12}}=\sqrt{\alpha} (\sqrt{2\pi })^{1-\alpha} \Bigl(x^{\frac{\alpha-1}{2}}\gamma^{\alpha x}\Bigr)^{-1} (\gamma\alpha)^{\alpha x}   e^{-\frac{\alpha }{12}}\\
   &\geq \sqrt{\alpha} e^{-\frac{\alpha }{12}}\frac{1}{(\sqrt{2\pi })^{\alpha-1}C_0(\alpha,\gamma)} (\gamma\alpha)^{\alpha x}=C_2(\alpha,\gamma)(\gamma\alpha)^{\alpha x},
\end{align*}
since $\bigl(\sqrt{2\pi}\bigr)^{\alpha-1}C_0(\alpha,\gamma)=D(\alpha,\gamma)$ for $\alpha>1$.
Similarly to the above, since $C_0(\alpha,\gamma)\to\infty$ as $\gamma\nearrow1$, we can assume that $\gamma$ is close enough to $1$ to ensure that $C_2(\alpha,\gamma)\leq1$ (for the given $\alpha)$. Then
\begin{align*}
    \psi(\alpha,\lambda) &=e^{-\alpha\lambda}+ e^{-\alpha\lambda}\sum_{i=1}^\infty\frac{\lambda^{\alpha i}}{(i!)^\alpha}
\geq e^{-\alpha\lambda}+e^{-\alpha\lambda}\sum_{i=1}^\infty\frac{\lambda^{\alpha i}}{\Gamma(\alpha i+1)}C_2(\alpha,\gamma)(\alpha \gamma)^{\alpha i}
\\&> C_2(\alpha,\gamma)e^{-\alpha\lambda}E_\alpha\left(\bigl(\alpha\gamma\lambda\bigr)^{\alpha } \right).
\end{align*}

We are going to show now that $\gamma$ can be chosen uniformly in $\alpha$. 

For $0<\alpha<1$, $C_1(\alpha,\gamma)\geq1$ iff
\[
\frac12\log\alpha +\frac1{12\alpha}+\frac{1-\alpha}{2}\left(\log \left(\frac{\pi(1-\alpha)}{e\alpha}\right)-\log(-\log\gamma)\right)\geq0,
\]
or equivalently,
\begin{equation}\label{eq:req1}
    \log(-\log\gamma)\leq \frac{\alpha}{2}\log\alpha+\frac{1}{6\alpha(1-\alpha)}+\log(1-\alpha)+\log\frac{\pi}e=:r_1(\alpha).
\end{equation}

For $\alpha>1$, $C_2(\alpha,\gamma)\leq 1$ iff
\[
\frac12\log\alpha -\frac{\alpha}{12}-\frac{\alpha-1}{2}\left(\log \left(\frac{\pi(\alpha-1)}{e\alpha}\right)-\log(-\log\gamma)\right)\leq 0.
\]
Denoting $\bar\alpha:=\frac1{\alpha}\in(0,1)$, we can rewrite this as follows
\begin{gather}
-\frac12\log\bar\alpha -\frac{1}{12\bar\alpha}-\frac{1-\bar\alpha}{2\bar\alpha}\left(\log \left(\frac{\pi(1-\bar\alpha)}{e}\right)-\log(-\log\gamma)\right)\leq 0;\notag\\
\log(-\log\gamma)\leq \frac{\bar\alpha}{1-\bar\alpha}\log\bar\alpha +\frac{1}{6(1-\bar\alpha)}+\log(1-\bar\alpha)+\log \frac{\pi}{e}=:r_2(\bar\alpha). \label{eq:req2}
\end{gather}
Since $\log\bar\alpha<0$, we have $\frac{\bar\alpha}{1-\bar\alpha}\log\bar\alpha <\frac{\bar\alpha}{2}\log\bar\alpha $; also, $\frac{1}{6(1-\bar\alpha)}<\frac{1}{6\bar\alpha(1-\bar\alpha)}$. Therefore,
\[
\inf_{0<\bar\alpha<1} r_2(\bar\alpha)\leq \inf_{0<\alpha<1} r_1(\alpha).
\]
To find the smaller infimum, note that
\[
r_2'(\bar\alpha)=\frac{1}{(1-\bar\alpha)^2}\log \bar\alpha+\frac{1}{1-\bar\alpha}+\frac{1}{6(1-\bar\alpha)^2}-\frac{1}{1-\bar\alpha}=
\frac{1}{(1-\bar\alpha)^2}\left( \log \bar\alpha+\frac16\right).
\]
Hence $r_2$ attains its (global) minimum at $\bar\alpha_*=\exp\bigl(-\frac16\bigr)\in(0,1)$, and it is equal to
\begin{align*}
    r_2(\bar\alpha_*)&=-\frac{\bar\alpha_*}{6(1-\bar\alpha_*)}+\frac{1}{6(1-\bar\alpha_*)}+\log(1-\bar\alpha_*)+\log \frac{\pi}{e}\\
    &=\log\bigl(1-e^{-\frac16}\bigr)+\log\frac{\pi}{e}+\frac16
    =\log\Bigl(\frac{\pi}{e}\bigl(e^{\frac16}-1\bigr)\Bigr).
\end{align*}
Therefore, the assumption
\begin{align*}
    \log(-\log\gamma) &\leq\log\Bigl(\frac{\pi}{e}\bigl(e^{\frac16}-1\bigr)\Bigr),\\
    \log\gamma&\geq-\frac{\pi}{e}\bigl(e^{\frac16}-1\bigr),\\
    \gamma&\geq\exp\left(-\frac{\pi}{e}\bigl(e^{\frac16}-1\bigr)\right)=\gamma_*.
\end{align*}
implies that both \eqref{eq:req1} and \eqref{eq:req2} hold for all $\alpha,\bar\alpha\in(0,1)$, and hence, then $C_1(\alpha,\gamma)\geq1$ for $0<\alpha<1$, and $C_2(\alpha,\gamma)\leq 1$ for $\alpha>1$.

The statement is proven.
\end{proof}

\begin{lemma}\label{le:psi-asymp}
For any $\alpha>0$,
\begin{equation*}\label{eq:psi-asymp}
    \psi(\alpha,\lambda)\sim \frac{1}{ \sqrt{\alpha}}(2\pi\lambda)^{\frac{1-\alpha}{2}}, \qquad\lambda\to\infty.
\end{equation*}
\end{lemma}

\begin{proof}
We are going to use the saddle point method. We have that
\[
\sum_{n=0}^\infty \frac{\lambda^{\alpha n}}{(n!)^\alpha}\sim \int_0^\infty \frac{\lambda^{\alpha x}}{\bigl(\Gamma(x+1)\bigr)^\alpha}\,dx, \qquad \lambda\to\infty.
\]
Next,
\[
\frac{\lambda^{\alpha x}}{\bigl(\Gamma(x+1)\bigr)^\alpha}=\exp\left( \alpha \bigl( x\log\lambda -\log \Gamma(x+1)\bigr)\right),
\]
and by Stirling's formula,
\[
x\log\lambda -\log \Gamma(x+1)  \sim x\log\lambda - x\log x+x-\frac12\log(2\pi x)=:f_\lambda(x), \qquad x\to\infty.
\]
Since
\[
f_\lambda'(x)=\log\lambda -\log x-\frac{1}{2x} ,
\]
the function $f_\lambda$ attains its maximum value at $x\approx\lambda$, as we can neglect $\frac{1}{2x}$ for $x\approx\lambda\to\infty$, so that $f'_\lambda(\lambda)\approx0$. Since $f_\lambda''(x)=-\frac{1}{x}+\frac{1}{2x^2}$, we get that, for $x\approx \lambda$, 
\[
f_\lambda(x)\approx f_\lambda(\lambda)+\frac12f_\lambda''(\lambda)(x-\lambda)^2=\lambda-\frac12\log(2\pi \lambda)-\frac1{2\lambda}\Bigl(1-\frac{1}{2\lambda}\Bigr) (x-\lambda)^2.
\]
Therefore, for $\lambda\to\infty$,
\begin{align*}
    \int_0^\infty \frac{\lambda^{\alpha x}}{\bigl(\Gamma(x+1)\bigr)^\alpha}\,dx&\sim
  \int_0^\infty e^{\alpha\bigl( f_\lambda(\lambda)+\frac12f_\lambda''(\lambda)(x-\lambda)^2\bigr)}\,dx\\&=\frac{e^{\alpha \lambda}}{(2\pi\lambda)^{\frac{\alpha}{2}}}\int_{-\lambda}^\infty e^{-\frac{\alpha}{2\lambda}\bigl(1-\frac{1}{2\lambda}\bigr)x^2}dx\\
  &= \frac{e^{\alpha \lambda}}{(2\pi\lambda)^{\frac{\alpha}{2}}}\sqrt{\frac{2\lambda}{\alpha\bigl(1-\frac{1}{2\lambda}\bigr)}}\int_{-\sqrt{\frac{\alpha\lambda}{2}\bigl(1-\frac{1}{2\lambda}\bigr)}}^\infty e^{-y^2}dy\\&\sim 
  \frac{e^{\alpha \lambda}}{(2\pi\lambda)^{\frac{\alpha}{2}}}\sqrt{\frac{2\lambda}{\alpha}}\int_{-\infty}^\infty e^{-y^2}dy
  =\frac{(2\pi\lambda)^{\frac{1-\alpha}{2}}}{ \sqrt{\alpha}}\,e^{\alpha \lambda}.
\end{align*}
By the definition \eqref{psifunction} of the function $\psi(\alpha,\lambda)$, this implies the statement.
\end{proof}

\section*{Acknowledgement}
YM is supported by The Swedish Foundation for Strategic Research, grant UKR24-0004, and by the Japan Science and
Technology Agency CREST, project reference number JPMJCR2115.
KR is supported by the Research Council of Finland, decision number 359815.
YM and KR acknowledge that the present research is carried out within the frame and support of the ToppForsk project no.~274410 of the Research Council of Norway with the title STORM: Stochastics for Time-Space Risk Models.

\end{document}